\documentclass[a4paper, 12pt]{amsart}
\usepackage[utf8]{inputenc}
\usepackage[english]{babel}
\usepackage{amsmath, amsthm, amssymb}
\usepackage{graphicx}
\usepackage[dvipsnames]{xcolor}
\usepackage{caption}
\usepackage{subcaption}
\usepackage{mathtools}
\usepackage{enumerate}
\usepackage{tikz}
\usepackage{tikz-cd}
\usepackage{textcomp}
\usepackage{enumitem}
\usepackage{hyperref}
\hypersetup{ 
	colorlinks = true,
	allcolors = red,
}
\usepackage[capitalize,noabbrev]{cleveref}

\usepackage[a4paper]{geometry}
\usepackage{titlesec}
\titleformat{\section}
{\normalfont\Large\bfseries}{\thesection}{1em}{}[{\titlerule[0.8pt]}]

\newtheorem{theorem}{Theorem}[section]
\newtheorem{proposition}[theorem]{Proposition}
\newtheorem{lemma}[theorem]{Lemma}
\newtheorem{corollary}[theorem]{Corollary}

\theoremstyle{definition}
\newtheorem{definition}[theorem]{Definition}

\newtheorem{remark}[theorem]{Remark}

\newtheorem{example}[theorem]{Example}

\newtheorem{question}[theorem]{Question}

\DeclareMathOperator{\GL}{GL}

\DeclareMathOperator{\B}{\mathcal{B}}
\DeclareMathOperator{\Aut}{Aut}
\DeclareMathOperator{\Sym}{Sym}

\DeclareMathOperator{\dG}{d}
\DeclareMathOperator{\mG}{m}

\DeclareMathOperator{\ZG}{Z}

\newcommand{\Z}{\mathbb{Z}}

\newcommand{\F}{\mathbb{F}}

\newcommand{\msc}[1]{\href{https://zbmath.org/classification/?q=#1}{#1}}
\newcommand{\gen}[1]{\langle #1\rangle}

\numberwithin{equation}{section}

\makeatletter
\renewcommand\subsection{\@startsection{subsection}{2}%
	\z@{.5\linespacing\@plus.7\linespacing}{-.5em}%
	{\normalfont \bfseries}}
\makeatother

\makeatletter
\@namedef{subjclassname@2020}{%
	\textup{2020} Mathematics Subject Classification}
\makeatother

\begin{document}

	\title{Independence and strong independence complexes of finite groups }
	\author[A.~Lucchini]{Andrea Lucchini}
	\address[Andrea Lucchini]{Universit\`a di Padova, Dipartimento di Matematica \lq\lq Tullio Levi-Civita\rq\rq}
	\email{lucchini@math.unipd.it}
 \author[M.~Stanojkovski]{Mima Stanojkovski}
	\address[Mima Stanojkovski]{Universit\`a di Trento, Dipartimento di Matematica}
	\email{mima.stanojkovski@unitn.it}

 \makeatletter
\@namedef{subjclassname@2020}{
 \textup{2020} Mathematics Subject Classification}
\makeatother

\subjclass[2020]{\msc{20D15}, \msc{20D30}, \msc{20D60}, \msc{05E45}, \msc{20F05}, \msc{20F16}}
\keywords{Independence complex, subgroup lattice, modular $p$-groups.}
	

	\begin{abstract}
Let $G$ be a finite group. In \cite{Cam24} two different concepts of independence (namely \emph{independence} and \emph{strong independence}) are introduced for the subsets of $G$, yielding to the definition of two simplicial complexes whose vertices are the elements of $G$. The \emph{strong independence complex} $\tilde\Sigma(G)$ turns out to be a subcomplex of the \emph{independence complex} $\Sigma(G)$. We discuss several invariant properties related to these complexes and ask a number of questions inspired by our results and the examples we construct. We study then the particular case of complexes on finite abelian groups, giving a characterization of the finite groups realizing them. 
In conclusion, answering a question of Peter Cameron, we classify all finite groups in which the two concepts of independence coincide. 
\end{abstract}

	\maketitle

    \section{Introduction}

\noindent    
There are a number of graphs whose vertex set is a group $G$ and whose edges reflect the structure of $G$ in some way. These include the commuting graph (first studied in 1955), the generating graph (from 1996), the power graph (from 2000), and the enhanced power graph (from 2007), all of which have
a considerable and growing literature; cf.\ \cref{sec:graphs}. In a recent paper Cameron made some preliminary observations towards an extension of current work on graphs defined on groups to simplicial complexes; cf.\ \cite{Cam24}. 

Recall that a simplicial complex $\Delta$ is a downward-closed collection of finite subsets
(called simplices or simplexes) of a set $X.$ We assume that every singleton of $X$ belongs to $\Delta$. For geometric reasons, a simplex of cardinality $k$ has
dimension $k-1$ and is referred to as a $(k-1)$-simplex.

Cameron has drawn particular attention to two complexes on groups defined in terms of independence, one of which had already been partially investigated by Pinckney \cite{Pin21} in her doctoral thesis. A subset $A$ of a group $G$ is called \emph{independent} if none of its elements can be expressed
as a word in the other elements and their inverses; or equivalently, there is no $a\in A$ such that $a\in\gen{A\setminus\{a\}}$.
The independence complex $\Sigma(G)$ of $G$ consists of all the independent subsets of $G.$  A subset $A$ of a group $G$ is called \emph{strongly independent} if no subgroup
of $G$ containing $A$ has fewer than $|A|$ generators. The strong independence
complex $\tilde \Sigma(G)$  of $G$ is the complex whose simplices are the strongly independent subsets   of $G$. With these definitions,  $\tilde\Sigma(G)$ turns out to be a subcomplex of $\Sigma(G)$.

Recall that the $k$-skeleton of a simplicial complex $\Delta$ consists of all the simplices of dimension at most $k$.
Thus, the $1$-skeleton of a simplicial complex is a graph. The 1-skeleton of the independence complex $\Sigma(G)$ is the complement of the power graph of $G$
, while the 1-skeleton of the  strong independence
complex $\tilde\Sigma(G)$ is the complement of the enhanced power graph of $G$ (see \cref{defgraph} and \cref{prop:sigma-to-graph}).


A natural question is how much the independence complex $\Sigma(G)$ and the strong independence complex $\tilde \Sigma(G)$ can tell us about the defining  group $G$.  We address this question when $G$ is a finite group. As we have just observed, from the 1-skeleton of $G$ we can deduce the power graph of $G.$ In some cases, this graph already contains relevant information about the structure of the group. In fact, it allows us to recover, for every integer $n$  the number of elements of $G$ of order $n.$ With this information, we can decide whether or not the group is, for instance, nilpotent or simple. Furthermore, a simple group is uniquely determined by its power graph, and therefore by its independence complex. 

In general, there exist non-isomorphic groups $G_1$ and $G_2$ such that their independence complexes are isomorphic. For example, it is not difficult to prove that if there exists a bijection between the elements of $G_1$ and those of $G_2$ that induces an isomorphism between their subgroup lattices, then this bijection induces also an isomorphism between their independence complexes (see \cref{prop:lattices}). The existence of such a bijection is guaranteed whenever there exists an isomorphism between the subgroup lattices of $G_1$ and  $G_2$ that maps subgroups of $G_1$ to subgroups of $G_2$ of the same order. Some   situations in which this can occur are described in \cref{sec:ind-lattices}, see in particular \cref{esempiouno} and \cref{ex:rottlaender}. On the other hand, all the examples available to us of non-isomorphic finite groups whose independence complexes are isomorphic are of this type. This brought us to conjecture that, given two finite groups $G_1$ and $G_2,$ the complex $\Sigma(G_1)$ is isomorphic to $\Sigma(G_2)$ if and only if there exists an isomorphism between their subgroup lattices  that preserves orders. The first main result of this paper is presented below as \cref{t:main} and confirms this conjecture in the particular case where $G_1$ is an abelian group. In order to state it we recall some definitions. A finite group $G$ is called \emph{modular} if its subgroup lattice is modular, i.e.\ $(H_1\vee H_2)\cap H_3 = H_1 \vee (H_2\cap H_3)$ for all subgroups $H_1,H_2,H_3$ of $G$ with $H_1\leq H_3.$ A finite group $G$ is called \emph{hamiltonian}  if it is nonabelian and all of its subgroups are normal.  

\begin{theorem}\label{t:main}Let $G_2$ be a finite group. Then there exists a finite abelian group $G_1$ such that $\Sigma(G_1)\cong \Sigma(G_2)$ if and only if $G_2$ is nilpotent
	and its Sylow subgroups are modular and nonhamiltonian, and this is equivalent to saying that there exists an index-preserving isomorphism between the subgroup lattice of $G_1$ and the subgroup lattice of $G_2.$
	\end{theorem}

\noindent
Finite modular $p$-groups have been classified by Iwasawa, and therefore the previous theorem provides a complete description of finite groups that have the same independence complex as an abelian group. The proof of \cref{t:main} is easily reduced to the case where $G_2$ is a $p$-group, and in the case of $p$-groups, our statement is equivalent to saying that two $p$-groups $G_1$ and $G_2$ have the same independence complex if and only if they have the same order and the same subgroup lattice. We are not aware of any arguments that allow us to derive information about the lattice of subgroups of $G$ directly from the knowledge of the independence complex $\Sigma(G).$ Indeed, to prove \cref{t:main}, we take a less direct route, aimed at proving that if a $p$-group $G_2$ has the same independence complex as a finite abelian $p$-group $G_1,$ then all of its 2-generated subgroups are metacyclic. This ensures that the subgroup lattice of $G_2$  is modular. To then exclude that $G_2$ is hamiltonian, it is sufficient to prove that $\Sigma(G_1) \cong \Sigma(G_2)$ implies that no subgroup of $G_2$ can be isomorphic to the quaternion group of order $8$.

By \cite[Corollary 3.1]{ZBM20}, if $G_1$ and $G_2$ are two finite groups, then the power graphs of $G_1$ and $G_2$ are isomorphic if and only if their enhanced power graphs are isomorphic. In particular, the 1-skeleton of $\Sigma(G_1)$ is isomorphic to the 1-skeleton of $\Sigma(G_2)$ if and only if the 1-skeleton of $\tilde\Sigma(G_1)$ is isomorphic to the 1-skeleton of $\tilde\Sigma(G_2).$ 
This prompts the natural question of whether the independence complexes are isomorphic if and only if the strong ones are.
We answer this question in the negative by providing examples of finite $p$-groups $G_1$ and $G_2$ such that the strong independence complexes $\tilde\Sigma(G_1)$ and $\tilde\Sigma(G_2)$ are isomorphic, but the independence complexes $\Sigma(G_1)$ and $\Sigma(G_2)$ are not (see \cref{ex:strong}). However, we also demonstrate that, if $G_2$
is abelian, then $\tilde\Sigma(G_2)\cong \tilde\Sigma(G_1)$ implies $\Sigma(G_2)\cong \Sigma(G_1).$
We don't have a direct proof of this statement, but it follows as a consequence of the following result, combined with \cref{t:main}.

\begin{theorem}\label{th2}
Let $G_1$ and $G_2$ be finite groups. If $G_1$ is abelian and $\tilde \Sigma(G_1)\cong \tilde \Sigma(G_2)$, then the subgroup lattices of $G_1$ and $G_2$ are isomorphic. 
\end{theorem}

\noindent
We do not know, in general, if $\Sigma(G_1)\cong \Sigma(G_2)$
 implies $\tilde\Sigma(G_1)\cong \tilde\Sigma(G_2)$.
 What makes this question difficult to answer is that, at present, we do not have a method to directly extract $\tilde\Sigma(G)$ from $\Sigma(G)$. 

In his paper Cameron asked the following question \cite[Qs.~3]{Cam24}: for which groups do the notions of independence and strong independence coincide? In \cref{sec:ind-strongind} we give a complete answer to this question for finite groups. If $G$ is a finite nilpotent group it can be easily seen that
$\Sigma(G)=\tilde\Sigma(G)$ if and only if $G$ is a monotone $p$-group, i.e.\ $G$ has the property that 
for any subgroup tower $K\leq H\leq G$, if $H$ is $d$-generated then so is $K$.
For $p$ an odd prime number, the monotone $p$-groups are classified by Mann in \cite{Mann05, Mann11}, while the monotone $2$-groups have been classified by Crestani and Menegazzo in \cite{CM12}. The case where $G$ is not nilpotent requires more work. It turns out that the property $\Sigma(G) = \tilde\Sigma(G)$ strongly constrains the structure of $G.$ Specifically, we have the following.

\begin{theorem}\label{t:non-nilp}
	Let $G$ be a finite non-nilpotent group. Then
$\Sigma(G) = \tilde\Sigma(G)$ if and only if $G=PQ$ is a Frobenius group such that:
		\begin{enumerate}[label=$(\alph*)$]
			\item $P$ is a normal abelian $p$-subgroup of $G$, where $p$ is a prime;
			\item $Q$ is a cyclic $q$-subgroup of $G$, where $q\neq p$ is a prime;
			\item if $\alpha$ is a generator of $Q$, then exactly one of the following holds:  
			\begin{itemize}
				\item there exists $m\in\Z$ coprime to $p$ such that, for every $x\in P$, one has $x^{\alpha}=x^m$;
				\item $P$ is homocyclic with $\dG(P)=2$ and $|\alpha|$ does not divide $p-1$.
			\end{itemize}
		\end{enumerate}
\end{theorem}
  
\smallskip
\noindent
\textbf{Notation.}
We use standard group theoretic notation and, for a group $G$, write
\begin{itemize}
    \item $H\leq G$ to indicate that $H$ is a subgroup of $G$,
    \item $|g|$ to denote the order of an element $g\in G$;
    \item $\dG(G)$ and $\mG(G)$ to denote respectively the minimum and the maximum cardinality of a minimal generating set of $G$;
    \item $|X|$ and $\gen{X}$ for the cardinality of and subgroup generated by $X\subseteq G$, respectively;
    \item $\ZG(G)$ for the center of $G$;
    \item $\Phi(G)$ for the Frattini subgroup of $G$;
    \item $(\gamma_i(G))_{i\geq 1}$ for the lower central series of $G$.
\end{itemize}
If $p$ is a prime number, $n$ a non-negative  integer, and $P$ a finite $p$-group, we write $\Omega_n(P)$ and $\mho_n(P)$ for the following subgroups:
\[
\Omega_n(P)=\gen{x\in P\mid x^{p^n}=1} \ \ \textup{ and } \ \ \mho_n(G)=\gen{x^{p^n} \mid x\in G}.
\]
We write $C_n$ to indicate a cyclic group of order $n$ and $\texttt{SmallGroup}(n,i)$ to indicate the group of order $n$ and index $i$ in the Small Groups Library of GAP \cite{GAP4}. We use the symbol $\cong$ to denote isomorphism between mathematical objects, when there is no doubt on which category they are considered in.

 \smallskip
 \smallskip
 \noindent
 \textbf{Acknowledgements.} This project is funded by the European Union – NextGenerationEU under the National Recovery and Resilience Plan (NRRP), Mission 4 Component 2 Investment 1.1 - Call PRIN 2022 No.\ 104 of February 2, 2022 of the Italian Ministry of University and Research; Project 2022PSTWLB (subject area: PE - Physical Sciences and Engineering)
 ``Group Theory and Applications''. The second author has been funded by the Italian program Rita Levi Montalcini for young researchers, Edition 2020.
 Both authors are members of the Indam group GNSAGA. We thank Eamonn O'Brien for his help in running computational experiments around Questions \ref{qs:lattices} and \ref{qs:strong} and the anonymous referee for
comments that helped us improve this paper’s exposition.

\section{Graphs and invariant properties}\label{sec:general-rmks}

\noindent
We start by recalling and expanding some of the concepts seen in the introduction. 

\begin{definition}
    Let $G$ be a finite group and let $X$ be a subset of $G$. Then $X$ is called
    \begin{itemize}
        \item \emph{independent} if, for every $x\in X$, one has that $x$ does not belong to $\gen{X\setminus\{x\}}$.
        \item \emph{strongly independent} if $X\subseteq H\leq G$ implies that $|X|\leq \dG(H)$. 
    \end{itemize}
\end{definition}

\begin{remark}\label{rmk:independences} 
    Let $G$ be a finite group and let $X$ be a subset of $G$. If $X$ is independent, then $X$ is a minimal generating set of $\gen{X}$. Assume now that $X$ is strongly independent. Then it follows from the definition that $\dG(\gen{X})=|X|\leq\dG(G)$. In other words, the cardinality of a strongly independent subset of $G$ cannot exceed the minimum number of generators of $G$. The strongly independent subsets of $G$ are in particular independent. 
\end{remark}

\noindent
The collection of independent subsets of a finite group $G$ forms a simplicial complex $\Sigma(G)$ called the \emph{independence complex} of $G$; cf.\ \cite[Prop.~2.1]{Cam24} or \cite[Lem.~2.4.2]{Pin21}. Analogously, the collection of strongly independence subsets of $G$ forms a subcomplex of $\Sigma(G)$, called the \emph{strong independence complex} of $G$ and which we denote with $\tilde\Sigma(G)$; cf.\ \cite[Prop.~2.3]{Cam24}. 
We note that, if $G_1$ and $G_2$ are finite groups and $\Sigma(G_1)\cong\Sigma(G_2)$ or $\tilde\Sigma(G_1)\cong\tilde\Sigma(G_2)$, then the number of $0$-simplices of the complexes are the same and so $G_1$ and $G_2$ have the same order. Moreover, it is also a straightforward consequence of the definitions that, if $G_1$ and $G_2$ have isomorphic (strong) independence complexes, then one is cyclic if and only if the other one is.

\subsection{Graphs associated to independence complexes}\label{sec:graphs}

\noindent
The following notions of graphs on groups are well-studied. See \cite{AKC13,BCDD24,cpg2,CGS09,KQ00,MKL22}, or the survey \cite{Cam22} for a broader outlook. 

\begin{definition}\label{defgraph}
    Let $G$ be a finite group. Then the 
    \begin{itemize}
        \item \emph{power graph} of $G$ is the undirected graph $\mathcal{P}(G)=(G, E)$ such that $\{x,y\}\subseteq G$ belongs to $E$ if and only if $x\neq y$ and $x\in\gen{y}$ or $y\in\gen{x}$. 
        \item \emph{directed power graph} of $G$ is the directed graph $\mathcal{G}(G)=(G, E)$ with $(x,y)\in G^2$ belonging to $E$ if and only if $x\neq y$ and $y\in\gen{x}$. 
        \item \emph{enhanced power graph} of $G$ is the undirected graph $\mathcal{E}(G)=(G, E)$ such that $\{x,y\}\subseteq G$ belongs to $E$ if and only if $x\neq y$ and $\gen{x,y}$ is cyclic. 
    \end{itemize}
\end{definition}

\noindent
The following is an easy consequence of the definitions of $\Sigma(G)$ and $\tilde\Sigma(G)$; see also Propositions 2.1 and 2.3 in \cite{Cam24}.

\begin{proposition}\label{prop:sigma-to-graph}
    Let $G$ be a finite group.
    Denote, moreover, by $\Sigma_1(G)$ and $\tilde\Sigma_1(G)$ the $1$-skeletons of $\Sigma(G)$ and $\tilde\Sigma(G)$, respectively. 
    Then the following hold: 
    \begin{itemize}
        \item The graph complement of $\Sigma_1(G)$ equals $\mathcal{P}(G)$.
        \item The graph complement of $\tilde\Sigma_1(G)$ equals $\mathcal{E}(G)$.
    \end{itemize}
\end{proposition}

\noindent
The following results are respectively \cite[Prop.~1]{cpg2} and a combination of \cite[Cor.~3.1]{ZBM20} and \cite[Prop.~1]{cpg2}. 

\begin{lemma}\label{lem:respect-orders}
    Let $G$ and $H$ be finite groups and let $\varphi:\mathcal{G}(G)\rightarrow\mathcal{G}(H)$ be an isomorphism of directed graphs. Then, for every $x\in G$, one has $|\varphi(x)|=|x|$. 
\end{lemma}

\begin{proposition}\label{prop:eq-graphs}
    Let $G$ and $H$ be finite groups. The following are equivalent:
    \begin{enumerate}[label=$(\arabic*)$]
        \item $\mathcal{P}(G)$ and $\mathcal{P}(H)$ are isomorphic;
        \item $\mathcal{G}(G)$ and $\mathcal{G}(H)$ are isomorphic;
        \item $\mathcal{E}(G)$ and $\mathcal{E}(H)$ are isomorphic.
    \end{enumerate}
    Moreover, if any of the previous equivalent conditions is satisfied, for every integer $n$, the following equality is satisfied:
    $    |\{g\in G : |g|=n\}|=|\{h\in H : |h|=n\}|$.
\end{proposition}

\noindent
Combining \cref{prop:sigma-to-graph} and \cref{prop:eq-graphs}, a natural question is whether there is any clear relationship between independence and strong independence complexes of a finite group. The next example provides an infinite family of pairs of $p$-groups with isomorphic strong independence complexes, but whose independence complexes are not isomorphic. We remind the reader that the \emph{rank} of a finite group $G$, denoted $\mathrm{rk}(G)$, is the smallest nonnegative integer $r$ such that every subgroup of $G$ can be generated by $r$ elements. It is easy to see that, when $p$ is a prime number and $G$ is a $p$-group, then the largest simplices in $\Sigma(G)$ have cardinality $\mathrm{rk}(G)$, i.e.\ dimension $\mathrm{rk}(G)-1$.    

\begin{example}\label{ex:strong}
Let $p\geq 5$ be a prime number and let 
\begin{itemize}
    \item $G_1$ be the unique group of maximal class of order $p^5$ and exponent $p$, up to isomorphism, containing a maximal abelian subgroup; cf.\ \cite[Thm.~4.3]{Bla58} (in Blackburn's classification take $\alpha=\beta=\gamma=\delta=0$).
    \item $G_2$ be the (unique) free $2$-generated, exponent $p$,  class $3$ group; then $|G_2|=p^5$. 
\end{itemize}
The group $G_1$ has clearly rank $4$, while $G_2$ has rank $3$ as $\gamma_2(G_2)$ is self-centralizing in $G_2$. Moreover, both groups are $2$-generated so, in the study of their strong independence complexes, we are only concerned with minimal generating sets of $2$-generated subgroups. The exponent of both groups being $p$, the number of $1$- and $2$-simplices is the same and equal to
\begin{itemize}
    \item $p^5-1$ for $1$-simplices,
    \item $\binom{p^5-1}{2}-\left(\frac{p^5-1}{p-1}\right)\binom{p-1}{2}$ for $2$-simplices.
\end{itemize}
The strong independence complexes of $G_1$ and $G_2$ are therefore isomorphic. 
\end{example}

\noindent
In the direction of better understanding the relationship between independence and strong independence, in \cite[Qs.~3]{Cam24} Cameron asks the following. 

\begin{question}\label{qs:Q3}
What are the finite groups $G$
for which $\Sigma(G)=\tilde\Sigma(G)$?    
\end{question}

\noindent
Thanks to \cref{rmk:independences}, \cref{qs:Q3} is equivalent to asking which are the finite groups $G$ for which the following is satisfied:
\begin{equation}\label{eq:Q3}
    H\leq K\leq G \quad \Longrightarrow \quad \mG(H)\leq \dG(K). 
\end{equation} 

\noindent
In \cite[Thm.~2.4]{Cam24}, Cameron remarks that in a finite group $G$ with $\Sigma(G)=\tilde\Sigma(G)$ every element has prime power order.
Moreover, if $p$ is a prime number and $G$ is a finite abelian $p$-group, then \cite[Thm.~2.5]{Cam24} ensures that $\Sigma(G)=\tilde\Sigma(G)$.  
Though we give a complete classification in \cref{sec:ind-strongind} of the groups from \cref{qs:Q3}, the following remains open. 

\begin{question}
    For finite groups $G_1$ and $G_2$, if $\Sigma(G_1)=\tilde\Sigma(G_1)$ and $\Sigma(G_2)\cong \Sigma(G_1)$, then does it necessarily hold that 
$\Sigma(G_2)=\tilde\Sigma(G_2)$? 
\end{question}

\noindent
The last question once again hints at the fact that, at this moment, we still have no method for telling which simplices of $\Sigma(G)$ belong to $\tilde\Sigma(G)$. 


\subsection{Graph-theoretic tools for independence complexes}

\noindent
In this section, we collect a number of results -- mostly from \cite{BP24}, \cite{cpg2}, and \cite{ZBM20} -- that we will employ to determine invariant properties of (strong) independence complexes. 
For $x$ in a group $G$, we write 
\begin{align*}
N(x) &=
    \{z\in G : \{x,z\} \textup{ is an edge in } \mathcal{P}(G)\}\cup \{x\}, \\
    N[x] & =\{z\in G : N(x)=N(z)\}, \\
    C[x] & =\{z\in G : \gen{x}=\gen{z}\},
\end{align*}
and, if $y$ is another element of $G$, we write 
    \begin{itemize}
        \item $x\equiv y$ if $N(x)=N(y)$, and 
        \item $x\approx y$ if $\gen{x}=\gen{y}$. 
    \end{itemize}
    Then $\equiv$ and $\approx$ are equivalence relations on $G$ such that $\equiv$-classes are unions of $\approx$-classes; cf.\ \cite{BP24,cpg2}. We will call \emph{$N$-class} (as is done in \cite{BP24}) an equivalence class with respect to $\equiv$ and \emph{$C$-class} an equivalence class with respect to $\approx$.
   We adopt the notation from \cite{cpg2}, but remark that 
$\approx$ is denoted $\diamond$ in \cite{BP24}. Given a non-empty subset $X$ of $G$, we will write 
$$N(X)=\bigcap_{x\in X}N(x) \ \ \textup{ and } \ \  \hat X=N(N(X)).$$
We define the collection of \emph{star vertices} of $G$ as
$$\mathcal{S}(G)=N[1]=\{z\in G : {N}(z)=G\}.$$ 
The next result is derived from \cite[Proposition 4]{BP24}.

\begin{lemma}\label{it:N1.5} Let $G$ be a finite noncyclic abelian group. Then $\mathcal{S}(G)=\{1\}$.
\end{lemma}

\noindent
The remaining part of the section collects a veriety of results on $N$- and $C$-classes.

\begin{lemma}\label{lem:eq-approx}
    Let $p$ be a prime number and let $G$ be a noncyclic abelian $p$-group. Then every $x\in G$ satisfies $N[x]=C[x]$. 
\end{lemma}

\begin{proof}
We show that every $N$-class of elements of $G$ is also a $C$-class.  To see this, let $x$ and $y$ in $G$ satisfy ${N}(x)={N}(y)$: we show that $\gen{x}=\gen{y}$. As $x$ and $y$ are adjacent in $\mathcal{P}(G)$, we assume for a contradiction that $|x|\neq |y|$ and, without loss of generality, that $y\in\gen{x^p}$. Let now $z\in G\setminus\gen{x}$ be of order $p$, where the existence of $z$ is guaranteed by the fact that $G$ is not cyclic. Then the following are satisfied:
\begin{itemize}
    \item $y$ and $xz$ are adjacent in $\mathcal{P}(G)$ because $y\in\gen{x^p}=\gen{(xz)^p}$, while
    \item $x$ and $xz$ are not adjacent in $\mathcal{P}(G)$ because $z\notin \gen{x}$. 
\end{itemize}
We deduce that $xz\in N(y)\setminus N(x)$. This contradicts the fact that $x$ and $y$ are in the same $N$-class and so the proof is complete.
\end{proof}


\noindent
The following terminology is taken from \cite{BP24}.

\begin{definition}
   Let $G$ be a finite group. An $N$-class $N=N[x]$ is 
\begin{itemize}
    \item \emph{of plain type} if it consists of a single $C$-class, i.e.\ if every $y\in N$ satisfies $x\approx y$; 
    \item \emph{of compound type} if it is not of plain type;
    \item a \emph{critical class} if
$\hat N = N \dot\cup \{1\}$ 
and there exist a prime number $p$ and an integer $r \geq  2$  with $|\hat N|=p^r$.
\end{itemize}
\end{definition}


\noindent
If $p$ is a prime number and $G$ is a finite noncyclic abelian $p$-group, then \cref{lem:eq-approx} guarantees that every $N$-class in $G$ is plain. 
In Cameron's paper \cite{cpg2}, classes of plain type ar referred to as classes of type (a), while those of compound type are called of type (b); cf.\ \cite[Prop.~5]{cpg2}.  
An $N$-class $N\neq \mathcal{S}(G)$ that is not critical  can be
immediately recognized as plain or compound by arithmetical considerations on $|N|$ and $|\hat N|$; see \cite[Prop.~17]{BP24}. 
On the other hand, when $\mathcal S(G)=\{1\},$ it is possible to recognize if a critical class is plain or compound by purely graph theoretic considerations. Indeed,  by \cite[Prop.~19]{BP24}, for a critical class $N=N[y]$, the following are equivalent: 
\begin{itemize}
    \item $N$ is of plain type;
    \item there exists $x\in G \setminus \hat N$ such that $x$ is adjacent to $y$ in $\mathcal{P}(G)$ and  $|N[x]|\leq |N|$.
\end{itemize}
The next result is now immediate.

\begin{lemma}\label{lem:plaintoplain}Assume that $G_1$ and $G_2$ are finite groups and that $\varphi: \mathcal P(G_1)\to \mathcal P(G_2)$ is an isomorphism between their power graphs. If $\mathcal S(G_1)=\{1\},$ then, for every $x\in G_1,$ the class $N[x]$ is of plain type if and only if the class $N[\varphi(x)]$  is of plain type.
\end{lemma}


\noindent
The following technical lemma is crucial for the ``local theory'' of independence complexes we develop in the next section.

\begin{lemma}\label{lem:classes} Let $G$ be a finite group and let $N$ be an $N$-class in $G$. The following hold:\begin{enumerate}[label=$(\arabic*)$]
  \item\label{it:N1} If $N$ is of plain type then  all elements of $N$ have the same order, say $n$, and $|N|=\phi(n)$, where $\phi$ denotes Euler's function.
\item\label{it:N2}If $N\neq \mathcal S(G)$ is of compound type then there exist a prime number $p$ and a positive integer $r$ such that $|\hat N|=p^r$ and the elements of $N$ have order a power of $p.$
\end{enumerate}\end{lemma}

\begin{proof}
The first point follows immediately from the definition of plain type class, while the second is a consequence of Propositions 10 and 13 from \cite{BP24}. 
\end{proof}


\subsection{Invariant properties for independence complexes}

\noindent
In this section we employ the theory outlined in the previous section to derive group invariants associated to (isomorphism classes of) independence complexes.  

\begin{proposition}\label{prop:iso-orders}
    Let $G_1$ and $G_2$ be finite groups with the property that  $\Sigma(G_1)\cong\Sigma(G_2)$ or $\tilde\Sigma(G_1)\cong\tilde\Sigma(G_2)$. Let $p$ be a prime number and let $y\in G_1$.
    Then the following hold:
    \begin{enumerate}[label=$(\arabic*)$]
        \item\label{it:ord1} If $n$ an integer, then $|\{x\in G_1 : |x|=n\}|=|\{y\in G_2 : |y|=n\}|$. 
        \item\label{it:ord2} If $G_1$ and $G_2$ are $p$-groups, then:
        \begin{itemize}
            \item $|\{x^p : x\in G_1\}|=|\{z^p : z\in G_2\}|$;
            \item there is $\tilde y\in G_2$ with $|y|=|\tilde y|$ and $|\{x\in G_1 : x^p=y\}|=|\{z\in G_2 : z^p=\tilde y\}|$.
        \end{itemize}
    \item\label{it:ord3} If $G_1$ and $G_2$ are $p$-groups such that $G_1$ is noncyclic abelian and $\varphi:G_1\rightarrow G_2$ is the bijection induced by the isomorphism of simplicial complexes, then $|\varphi(y)|=|y|$.
    \end{enumerate}
\end{proposition}

\begin{proof}
\ref{it:ord1} This is a direct consequence of \cref{prop:sigma-to-graph} and
\cref{prop:eq-graphs}.

\noindent
\ref{it:ord2}
We start by proving the first bullet point. 
Again by \cref{prop:sigma-to-graph} and
\cref{prop:eq-graphs}, the $1$-skeletons of the complexes are isomorphic and $\mathcal{G}(G_1)\cong\mathcal{G}(G_2)$. The groups being $p$-groups, one can easily read off $p$-th powers from the directed power graphs. 

We now proceed with the second bullet point and, for this, let $f:\mathcal{G}(G_1)\rightarrow\mathcal{G}(G_2)$ be an isomorphism of directed graphs. Call $\tilde y=f(y)$ and note that $|\tilde y|=|y|$ as isomorphisms of directed graphs respect element orders; cf.\ \cref{lem:respect-orders}. It is not difficult to extract $X=\{x\in G_1 : x^p=y\}$ from $\mathcal{G}(G_1)$ and so $f(X)=\{z\in G_2 : z^p=\tilde y\}$.

\noindent\ref{it:ord3} 
By slight abuse of notation we call $\varphi$ also the graph isomorphism $\mathcal{P}(G_1)\rightarrow\mathcal{P}(G_2)$ that is induced by the isomorphism of complexes.
Thanks to \cref{lem:classes}, the $N$-class $\mathcal{S}(G_1)$ is trivial and therefore so is $\mathcal{S}(G_2)$. 
Let now $N=N[x]\neq \{1\}$ be a nontrivial $N$-class in $G_1$, which is of plain type thanks to \cref{lem:eq-approx}. 
By \cref{lem:plaintoplain} 
the class $\varphi(N)$ is also of plain type. Now \cref{lem:classes}\ref{it:N1}  ensures that that $\phi(|\varphi(x)|)=|\varphi(N[x])|=|N[x]|=\phi(|x|),$
and, since $|\varphi(x)| $and $|x|$ are $p$-powers, we conclude that $|\varphi(x)|=|x|$. 
\end{proof}

\begin{remark}\label{rmk:simple}
The spectrum $\omega(G)$ of a finite group $G$ is the set of
its element orders.
Thanks to \cref{prop:iso-orders}\ref{it:ord1}, for every integer $n$, from the (strong) independence complex of $G$ we deduce the number of elements of $G$ of order $n$ and, in particular, also $\omega(G)$. The main theorem of \cite{vgm} states that if $L$ is a finite simple group and $G$ is finite
with $\omega(G) = \omega(L)$ and $|G| = |L|$, then $G$ is isomorphic to $L.$ Hence a finite simple group $G$ can be recognized from its (strong) independence complex. 
\end{remark}

\noindent
Because of \cref{rmk:simple}, the case of finite simple groups is completely settled, so we focus on the case of nilpotent groups in \cref{sec:ind}. 
In this direction, we show that nilpotency is a property that is preserved by both isomorphism of complexes and, in view of \cref{sec:ind-lattices}, isomorphism of subgroup lattices. 



\begin{proposition}\label{prop:nilpotent}
    Let $G$ and $H$ be finite groups with $H$ nilpotent. 
    The following hold:
    \begin{enumerate}[label=$(\arabic*)$]
        \item\label{it:propS1} If $G$ and $H$ have isomorphic independence complexes, then $G$ is nilpotent.
        \item\label{it:propS2} If $G$ and $H$ have isomorphic strong independence complexes, then $G$ is nilpotent.
        \item\label{it:propL1} If $G$ and $H$ have the same orders and isomorphic subgroup lattices, then $G$ is nilpotent.
    \end{enumerate}
\end{proposition}

\begin{proof}
\ref{it:propS1}-\ref{it:propS2} The (strong) independence complexes of $G$ and $H$ being isomorphic, \cref{prop:eq-graphs} ensures that $G$ and $H$ have the same orders and isomorphic power graphs.  Then  \cite[Corollary 3.2]{MirSca} yields that $G$ is nilpotent. 

\noindent
   \ref{it:propL1} This is a a consequence of \cite[Thm.~1.6.5]{schmidt}.
\end{proof}

\begin{remark}
    Let $G$ be a finite group and call a simplex $Y$ in $\Sigma(G)$ a \emph{generating simplex} if $\gen{Y}=G$. Assume that we are able, for any simplex $Y$ in $\Sigma(G)$, to tell whether $Y$ is a generating simplex. Then we also know which subsets of $G$ are generating sets: indeed, by the minimality of independent sets, $X\subseteq G$ is a generating set if and only if it contains a generating simplex of $\Sigma(G)$. Knowing the generating simplices in $G$ implies, in the language of \cite{Luc21}, being in Situation 1 and \emph{having enough information on the generating properties} of $G$ to deduce, for instance:
    \begin{itemize}
        \item whether $G$ is nilpotent, supersolvable, perfect, \cite[Prop.~6]{Luc21};
        \item whether $G$ is solvable, \cite[Thm.~9]{Luc21};
        \item the number of non-Frattini resp.\ non-abelian factors in a chief series of $G$, \cite[Thm.~10]{Luc21}. 
    \end{itemize}
    In other words, if we could identify generating simplices in $\Sigma(G)$, 
    we could, for instance, derive directly from $\Sigma(G)$ whether $G$ is nilpotent (without the need of a nilpotent representative $H$ for $\Sigma(H)\cong\Sigma(G)$ as \cref{prop:nilpotent}\ref{it:propS1} would require). 
\end{remark}

\noindent
The following proposition ensures that, for noncyclic groups having the same independence complex as a finite abelian group, the information provided by the complex can be localized at the primes dividing the number of its vertices. 

\begin{proposition}\label{prop:iso-general}
        Let $G_1$ and $G_2$ be finite groups with $G_1$ abelian and noncyclic. 
         Assume  $\Sigma(G_1)\cong\Sigma(G_2)$ and let $\varphi:G_1\rightarrow G_2$ be the bijection induced by the isomorphism of complexes. 
         If $p$ is a prime number and  $S$ is a Sylow $p$-subgroup of $G_1$, then  $\varphi(S)$ is a Sylow $p$-subgroup of $G_2$. 
\end{proposition}

\begin{proof}
We recall that an isomorphism $\Sigma(G_1)\rightarrow\Sigma(G_2)$ induces, by \cref{prop:sigma-to-graph}, an isomorphism $\mathcal{P}(G_1)\rightarrow\mathcal{P}(G_2)$, which we also denote by $\varphi$. Moreover, from \cref{prop:nilpotent}\ref{it:propS1}, we know that $G_2$ is nilpotent and, from \cref{it:N1.5}, that $\mathcal{S}(G_1)=\{1\}$, so that also $\mathcal{S}(G_2)=\{1\}$. It follows from \cref{lem:plaintoplain}
that if $N$  is a class of plain type (respectively compound type), then the class $\varphi(N)$ is also of plain type (respectively compound type).
Let now $p$ be a prime number, $S$ a Sylow $p$-subgroup of $G_1$, and $x\in S$. We show that the order of $\varphi(x)$ is necessarily a power of $p$. Since $\mathcal S(G_1)=\{1\}$ and
$\mathcal S(G_2)=\{1\}$, we clearly have $\varphi(1)=1$ and so we take $x\neq 1$.

Assume first that $N[x]$
is of compound type. Then, by \cref{lem:classes}\ref{it:N2}, the elements of $N[x]$ belong to $S$ and
$|\widehat{N[x]}|=p^n$ for some integer $n$.
Since $\varphi$ is an isomorphism, we have that
$\varphi(N[x])=N[\varphi(x)]$ is of compound type  and 
$|\widehat{N[\varphi(x)]}|=
|\varphi(\widehat{N[x]})|=p^n.$ If follows from \cref{lem:classes}\ref{it:N2} that the elements of $\varphi(N[x])$, and in particular $\varphi(x),$ all have order a power of $p.$

Assume now that $N[x]$ is of plain type.  \cref{lem:classes}\ref{it:N1} yields that $|N[x]|=\phi(|x|)$
and since $\varphi$ is a graph isomorphism, we have
that  $\varphi(N[x])$ is a plain class of size $\phi(|x|)$ and therefore 
\cref{lem:classes}\ref{it:N1} implies that $\phi(|\varphi(x)|)=|\varphi(N[x])|=\phi(|x|).$
The last equality implies that either $|\varphi(x)|=|x|$ or $|\varphi(x)|=2|x|$, in which case $p$ is odd. 
If $p=2$, then we clearly have that $\varphi(x)$ has order a power of $p$ and so the case of elements of order a power of $2$ is completely settled.
We assume now that $p$
is odd and, for a contradiction, that $|\varphi(x)|=2|x|$. We let $y\in G_1$ be such that $\varphi(y)=\varphi(x)^{|x|}$. Then $\varphi(y)$ and $\varphi(x)$ are connected by an edge in the power graph of $G_2$ and so, $x$ and $y$ are adjacent in $\mathcal{P}(G_1)$.
Since $|\varphi(y)|=2$, applying the inverse of $\varphi$, we deduce that $|y|=|\varphi^{-1}(\varphi(y))|$ is a power of $2$ and so $x$ and $y$ have coprime orders. This is a contradiction to $x$ and $y$ being adjacent in $\mathcal{P}(G_1)$.

We have proven that, in any case, $|\varphi(x)|$ is a $p$-th power and therefore, $x$ being arbitrary, $|G_1|=|G_2|$ implies that $\varphi(S)$ is a Sylow $p$-subgroup of $G_2$.
\end{proof}

\subsection{Graphs and invariant properties for strong independence complexes}

\noindent
We now proceed towards proving an analogue of \cref{prop:iso-general} for strong independence complexes; cf.\ \cref{prop:iso-general-tilde}. In this direction, the next lemma is an analogue of \cref{lem:eq-approx} for enhanced power graphs.

\begin{lemma}\label{lem:strong-to-p}
    Let $G$ be a finite abelian group and let $X\subseteq G$. For an element $g\in G$ and a prime $p$, if $m_p$ is the largest positive divisor of $|g|$ that is coprime to $p$, write $g_p$ for $g^{m_p}$. 
The following are equivalent:
\begin{enumerate}[label=$(\arabic*)$]
    \item\label{it:X} $X$ is strongly independent in $G$;
    \item\label{it:Xp} there exists a prime $p$ such that $X_p=\{x_p : x \in X\}$ is strongly independent in $G$ and $|X_p|=|X|$. 
\end{enumerate}
\end{lemma}

\begin{proof}
    \ref{it:Xp} $\implies$ \ref{it:X} Let $p$ be such that $X_p$ is strongly independent and $|X|=|X_p|$. Let, moreover, $H$ be a subgroup of $G$ containing $X$. Then $H$ also contains $X_p$, which is strongly independent in $G$. In particular $|X|=|X_p|\leq \dG(H)$ and so $X$ is strongly independent. 

    \ref{it:X} $\implies$ \ref{it:Xp} Assume that $X$ is strongly independent.  Write $K=\gen{X}$ and note that $\dG(K)=|X|$; cf.\ \cref{rmk:independences}. Moreover, for every prime $p$, write $K_p$ for the Sylow $p$-subgroup of $K$. Since $G$ is finite abelian, there exists a prime $p$ such that $\dG(K_p)=\dG(K)$. Since $X_p$ generates $K_p$ and clearly $|X_p|\leq|X|$, we deduce that $|X_p|=|X|$. Now, $G$ being abelian, every subgroup $H$ of $G$ such that $X_p\subseteq H$ satisfies $|X_p|\leq \dG(H)$ and so $X_p$ is strongly independent. 
\end{proof}

\noindent 
In the next proposition, \ref{it:tilde1} is equivalent to saying that  $x$ and $y$ have same neighbours in the enhanced power graph (see also \cite[Lem.~3.2]{ZBM20}), while \ref{it:tilde2} is the same as saying that $x$ and $y$ have the same neighbours in the strong independence complex of the group. 

\begin{proposition}\label{prop:eq-approx-tilde}
Let $G$ be a finite group and let $x$ and $y$ be elements of $G$. Define:
\begin{enumerate}[label=$(\arabic*)$]
    \item\label{it:tilde1} $x\sim_c y$ if $x$ and $y$ belong to the same maximal cyclic subgroups of $G$;
    \item\label{it:tilde2} $x\asymp y$ if for every $T\subseteq G$ the following are equivalent:
    \begin{itemize}
        \item $\{x\}\cup T$ is a strong independence subset of $G$;
         \item $\{y\}\cup T$ is a strong independence subset of $G$.
    \end{itemize}
\end{enumerate}
Then $\sim_c$ and $\asymp$ are equivalence relations on $G$. Furthermore, if $G$ is abelian, then their equivalence classes coincide. 
\end{proposition}

\begin{proof}
It is easy to see that $\sim_c$ and $\asymp$ are equivalence classes on $G$ by observing that they can both be interpreted as ``having the same neighbors'' in a given simplicial complex, namely the enhanced graph for $\sim_c$ and the strong independence complex for $\asymp$. Assume now that $G$ is abelian. We will show that 
\[
x\sim_c y \ \ \Longleftrightarrow \ \ x\asymp y.
\]
We start by showing the implication from right to left. To this end, let $C=\gen{c}$ be a maximal cyclic subgroup of $G$. Then, by the maximality of $C$, we have
\[
x\in C \ \ \Longleftrightarrow \ \ \{x,c\} \textup{ is not strongly independent}.
\]
Assuming that $x$ and $y$ have the same neighbours in $\tilde\Sigma(G)$, we then have that $x\in C$ is equivalent to $\{y,c\}$ not being strongly independent, and so, by the maximality of $C$, to $y$ being an element of $C$.

We now prove the other implication. To this end, assume that $x$ and $y$ belong to the same maximal cyclic subgroups. 
Let now $g_1,\ldots,g_t\in G$ be such that $X=\{x,g_1,\ldots,g_t\}$ is strongly independent: we show that $Y=\{y,g_1,\ldots,g_t\}$ 
 is also strongly independent. Write $H=\gen{g_1,\ldots,g_t}$ and use the same notation from \cref{lem:strong-to-p}. Since $G$ is abelian, by the same lemma, there exists a prime number $p$ such that $X_p=\{x_p,g_{1,p}, \ldots, g_{t,p}\}$ 
 is strongly independent in $G$ and $|X|=|X_p|$. Fix such a prime $p$: we prove that $Y_p=\{y_p,g_{1,p}, \ldots, g_{t,p}\}$ is strongly independent in $G$.
Since $x$ and $y$ belong to the same maximal cyclic subgroups, $x_p\in\gen{y_p}$ or $y_p\in\gen{x_p}$. If $x_p\in\gen{y_p}$, it is clear that $Y_p$ is also strongly independent. We assume therefore that $\gen{y_p}$ is properly contained in $\gen{x_p}$, in other words $y_p\in\gen{x_p^p}$. Assume for a contradiction that $Y_p$ is not strongly independent. From $y_p\in\gen{x_p^p}\subseteq \gen{x_p}$ we deduce then that $y_p\in H$. If $y_p\notin \Phi(H)=\mho_1(H)$, then there exists $i\in\{1,\ldots,t\}$ such that $$g_{i,p}\in\gen{y_p,g_{1,p},\ldots,g_{i-1,p},g_{i+1,p},\ldots,g_{t,p}}\subseteq \gen{x_p,g_{1,p},\ldots,g_{i-1,p},g_{i+1,p},\ldots,g_{t,p}},$$
contradicting the strong independence of $X_p$. We have thus $y_p\in\mho_1(H)$. Let $h\in H\setminus \Phi(H)$ and $s>0$ be such that $h^{p^s}=y_p$. Since $y$ and $x$ are contained in the same maximal cyclic subgroups, there exists $a\in G$ such that $x_p$ and $h$ both belong to $\gen{a_p}$. Since $X_p$ is strongly independent, we deduce that $x_p$ does not belong to $\gen{h}\subseteq H$ and so $h\in\gen{x_p^p}$. The fact that $h\notin \Phi(H)$ again contradicts the strong independence of $X_p$. Indeed, if $h=g_1^{\alpha_1}\cdots g_t^{\alpha_t}$ and $\alpha_j\not\equiv 0\bmod p$, then 
$$g_{j,p}\in\gen{h, g_{1,p},\ldots,g_{j-1,p},g_{j+1,p},\ldots,g_{t,p}}\subseteq \gen{x_p,g_{1,p},\ldots,g_{j-1,p},g_{j+1,p},\ldots,g_{t,p}};$$
contradiction. 
This shows that $Y_p$ is strongly independent in $G$. Since $G$ is abelian, we conclude from \cref{lem:strong-to-p} that 
$Y$ is strongly independent in $G$.
\end{proof}

\begin{proposition}\label{prop:iso-general-tilde}
        Let $G_1$ and $G_2$ be finite groups with $G_1$ abelian. 
        Assume that  $\tilde\Sigma(G_1)\cong\tilde\Sigma(G_2)$ and let $\varphi:G_1\rightarrow G_2$ be the bijection induced by the isomorphism of complexes. If $p$ is a prime and  $S$ is a Sylow $p$-subgroup of $G_1$, then  $\varphi(S)$ is a Sylow $p$-subgroup of $G_2$. 
\end{proposition}

\begin{proof}
We recall that an isomorphism $\tilde\Sigma(G_1)\rightarrow\tilde\Sigma(G_2)$ induces, by \cref{prop:sigma-to-graph}, an isomorphism $\mathcal{E}(G_1)\rightarrow\mathcal{E}(G_2)$, which we also denote by $\varphi$. Moreover, from \cref{prop:nilpotent}\ref{it:propS2}, we know that $G_2$ is nilpotent. 
Let now $p$ be a prime number, $S$ a Sylow $p$-subgroup of $G_1$, and $x\in S$. 

With the notation from \cref{prop:eq-approx-tilde}, we write $R[x]$  for the $\asymp$-class, which coincides with the $\sim_c$-class, of $x$ in $G_1$. It follows that the  $\asymp$-class and the $\sim_c$-class of $\varphi(x)$ in $G_2$ also coincide. Denote this class by $R[\varphi(x)]$. As a consequence of \cite[Thm.~3.2]{ZBM20} there exists a bijection $f:G_1\rightarrow G_2$ such that 
 for each $y\in G_1$, the orders of $y$ and $f(y)$ are the same, and $f(R[y])=R[f(y)]$.  
Thanks to \cref{prop:eq-approx-tilde}, we assume without loss of generality that $\varphi=f$. It follows in particular that $\varphi(x)$ has order a power of $p$ and thus, $x$ being arbitrary, $\varphi(S)$ is contained in a Sylow $p$-subgroup of $G_2$. Since $|G_1|=|G_2|$ and $G_2$ is nilpotent, we conclude that $\varphi(S)$ is a Sylow $p$-subgroup of $G_2$.
\end{proof}

\subsection{Independence complexes and subgroup lattices}\label{sec:ind-lattices}

\noindent
In the following we refer to \cite{schmidt} for the knowledge on subgroup lattices and isomorphisms between them; these are called \emph{projectivities} in \cite{schmidt}. We use the standard notation from the textbook without introduction. For a finite group $G$, we will denote by $\mathcal{L}(G)$ its subgroup lattice. If $G_1$ and $G_2$ are finite groups and $\alpha:\mathcal{L}(G_1)\rightarrow\mathcal{L}(G_2)$ is a lattice isomorphism, $\alpha$ is called \emph{index-preserving} if it preserves indices of subgroups, i.e., the groups being finite, if every subgroup $H$ of $G$ satisfies $|\alpha(H)|=|H|$; cf.\ \cite[\S.~4.2]{schmidt}.

\begin{proposition}\label{prop:lattices}
    Let $G_1$ and $G_2$ be finite groups. An index-preserving lattice isomorphism $\mathcal{L}(G_1)\rightarrow\mathcal{L}(G_2)$ induces an isomorphism $\Sigma(G_1)\rightarrow\Sigma(G_2)$ and, moreover, this isomorphism maps $\tilde\Sigma(G_1)$ to $\tilde\Sigma(G_2)$. 
\end{proposition}

\begin{proof}
Let $\alpha:\mathcal L(G_1) \to \mathcal L(G_2)$
be an index-preserving isomorphism, so in particular $|G_1|=|G_2|$. We start by showing that $\Sigma(G_1)$ and $\Sigma(G_2)$ are isomorphic through a map induced by $\alpha$.
To this end, let $H\leq G_1$ be cyclic and so $\alpha(H)$ is also cyclic, as a consequence of \cite[Thm.~1.2.3]{schmidt}. Moreover,  $|\alpha(H)|=|H|$ because $\alpha$ is index-preserving. 
In other words, $\alpha$ induces an order-preserving bijection between the cyclic subgroups of $G_1$ and the cyclic subgroups of $G_2$. 
Let $\gamma: G_1\to G_2$ be a bijection compatible with $\alpha$, i.e.\ satisfying, for every $g\in G$, that $\langle \gamma(g)\rangle=\alpha(\langle g\rangle)$.  We claim that 
$\gamma$ induces an isomorphism $\Sigma(G_1)\rightarrow\Sigma(G_2).$ For this, let $\{g_1,\dots,g_t\}$ be a simplex
in $\Sigma(G_1)$: we will prove that $\{\gamma(g_1),\dots,\gamma(g_t)\}$ is an independent subset of $G_2.$ 
Assume, for a contradiction, that it is not and, without loss of generality, that $\gamma(g_t) \in \langle \gamma(g_1), \dots, \gamma(g_{t-1})\rangle.$ 
Then
$$\begin{aligned}
\alpha(\langle g_t\rangle)&=
\langle \gamma(g_t)\rangle \leq \langle \gamma(g_1),\dots,\gamma(g_{t-1})\rangle =
\langle \gamma(g_1)\rangle \vee \dots \vee \langle \gamma(g_{t-1})\rangle = \\&=\alpha(\langle g_1\rangle) \vee \dots \vee \alpha(\langle g_{t-1}\rangle)=\alpha(\langle g_1\rangle \vee \dots \vee \langle g_{t-1}\rangle)=\alpha(\gen{g_1,\ldots,g_{t-1}}).
\end{aligned}$$
Thus $\langle g_t\rangle \leq \langle g_1\dots g_{t-1}\rangle,$ and
$\{g_1,\dots,g_t\}$ is not independent. This is a contradiction and so we deduce that $\Sigma(G_1)$ and $\Sigma(G_2)$ are isomorphic via $\gamma$. 

We conclude by showing that $\gamma$ maps strong independent sets of $G_1$ to strong independent sets of $G_2$. For this, assume that $X=\{g_1,\dots,g_t\}$ is a strongly independent subset of $G_1$ and let $K$ be a subgroup of $G_2$ containing $\gamma(X)$. Write $K=\alpha(H)$, so that $\gen{X}\leq H$. Since $X$ is strongly independent, $|\gamma(X)|=|X|\leq\dG(H)$. 
Moreover, as $\alpha$ induces an isomorphism $\mathcal{L}(H)\rightarrow\mathcal{L}(K)$ and $\alpha$ preserves cyclicity, the minimum number of cyclic subgroups needed to generate $H$ and $K$ is the same, i.e.\ $\dG(H)=\dG(K)$. It follows that $|\gamma(X)|\leq \dG(K)$ and, $K$ being arbitrary, $\gamma(X)$ is strongly independent in $G_2$.
\end{proof}

\noindent
Combining \cref{prop:lattices} with the fact that every lattice isomorphism involving a finite $p$-group is index-preserving (cf.\ \cite[Lem.~4.2.1]{schmidt}) yields the following:

\begin{corollary}\label{cor:lattices}
Let $p$ be a prime number. 
    If $G_1$ and $G_2$ are $p$-groups with isomorphic subgroup lattices, then $\Sigma(G_1)\cong\Sigma(G_2)$ and $\tilde\Sigma(G_1)\cong\tilde\Sigma(G_2)$. 
\end{corollary}

\noindent
The most well-known examples of pairs of prime-power-order groups with isomorphic subgroup lattices involve a finite abelian $p$-group and a modular $p$-group of the same order that is not hamiltonian; cf.\ \cite[Thm.~2.5.9]{schmidt}. This construction is due to Baer \cite{Baer44} and we give an example below. 

\begin{example}\label{esempiouno}
    Let $G_1$ be an abelian group of order $27$ and exponent $9$ and let $G_2$ be non-abelian, with the same order and the same exponent. Then $G_1\cong C_9\times C_3$ and $G_2$ is the unique extraspecial group of order $27$ and exponent $9$.
    It is not difficult to show that $G_1$ and $G_2$ have isomorphic subgroup lattices and correspond to each other via Baer's construction. Thanks to \cref{cor:lattices}, the independence complexes of $G_1$ and $G_2$ are isomorphic. 
\end{example}

\noindent
\cref{esempiouno} also fits within the framework of \cite{Car79}, where more examples involving groups of maximal class can be found. For other examples involving $p$-groups, without necessarily listing abelian representatives in the class, see for instance \cite{DH75}. 

\begin{remark}
    Let $p$ be a prime number and let $G$ be a finite $p$-group. For every positive integer $r$ set
    \[
    \ell_r(G)=\sum_{H\leq G,\, \dG(H)=r} \left(\frac{|H|}{p^r}\right)^r=\sum_{H\leq G,\, \dG(H)=r} |\Phi(H)|^r.
    \]
    Up to a factor $|\GL(r,p)|/|\Sym(r)|$, the summand corresponding to a fixed $H\leq G$ with $\dG(H)=r$ counts the possible $r$-element sets generating $H$. In particular, 
    any two $p$-groups $G_1,G_2$ with isomorphic independence complexes satisfy $(\ell_r(G_1))_{r>0}=(\ell_r(G_2))_{r>0}$.  The smallest pair of groups with identical $\ell$-vectors but non-isomorphic independence complexes is given by \texttt{SmallGroup}(16,2) and \texttt{SmallGroup}(16,4). Nonetheless, computational experiments seem to suggest that $\ell$-vectors are a good first invariant to test non-isomorphism of independence complexes.
\end{remark}

\noindent
We proceed by considering the case of groups whose order has more than one prime divisor. 
The following example can be found in \cite[\S 6]{Rottlaender} (see also \cite[Ex.~5.6.8]{schmidt}) and satisfies the hypotheses of \cref{prop:lattices}.

\begin{example}\label{ex:rottlaender}
Let $p$ and $q$ be prime numbers such that $q\geq 5$ and $q$ divides $p-1$. Let $N=\F_p^2$ be the standard $2$-dimensional vector space over the field of $p$ elements. Let $\omega$ be a generator of the $q$-torsion subgroup of $\F_p^{\times}$ and let $\lambda\in\{2,\ldots, q-1\}$. 
Then 
\[
H_{\lambda}=\left\langle\begin{pmatrix}
    \omega & 0 \\ 0 & \omega^{\lambda}
\end{pmatrix}\right\rangle\leq\GL(2,p)
\]
has order $q$ and $G_{\lambda}=N\rtimes H_{\lambda}$ is a non-abelian group of order $p^2q$. Moreover, $N$ is a normal Sylow $p$-subgroup of $G_{\lambda}$. If $\mu\in\{2,\ldots,q-1\}$, then there exists an index preserving isomorphism $\mathcal{L}(G_{\lambda})\rightarrow\mathcal{L}(G_{\mu})$; cf.\  \cite[Thm.~4.1.8]{schmidt}. Morover, if $\mu\neq\lambda$, then $G_{\lambda}\cong G_{\mu}$ if and only if $\lambda\mu=1$ (and the isomorphism is induced by swapping the eigenspaces in $N$).
In summary, this general construction yields $((q-2)+1)/2=(q-1)/2$ pairwise non-isomorphic groups with isomorphic independence complexes and so we obtain arbitrary large families as $q$ grows. 
The smallest order in this collection of examples is realized for $q=5$ and $p=11$, yielding two non-isomorphic groups of order $605$. These correspond to \texttt{SmallGroup}(605,5) and \texttt{SmallGroup}(605,6) and can also be found in \cite[\S~2]{LMRD17}.
\end{example}

\noindent
The next example shows that, in \cref{prop:lattices}, the assumption that the isomorphism of lattices preserves indices is really necessary. It exhibits two groups of the same orders whose lattices of subgroups are isomorphic, but not via an index preserving map.

\begin{example}\label{ex:42}
    Let  
    $H$ be a cyclic group of order $6$ and define $N=\mathbb{F}_{7}$ to be the (underlying group of the) field of $7$ elements. Note that $\Aut(N)\cong\F_7^{\times}$ and let $\varphi_1:H\rightarrow\Aut(N)$ and $\varphi_2:H\rightarrow\Aut(N)$ be homomorphisms such that $|\varphi_1(H)|=2$ and $|\varphi_2(H)|=3$.
Then the groups $G_1=N\rtimes_{\varphi_1}H$ and $G_2=N\rtimes_{\varphi_2}H$ are non-isomorphic groups of order $42$, with $G_1\cong\texttt{SmallGroup}(42,4)$ and $G_2\cong\texttt{SmallGroup}(42,2)$. The  subgroup lattices of $G_1$ and $G_2$ are isomorphic. However, the first group has seven subgroups of order $2$ and one of order $3$, while in the second the situation is reversed. In particular, the isomorphism between the lattices is not index-preserving. Moreover, since the power graphs of the groups are distinct, \cref{prop:iso-orders}\ref{it:ord1} yields that the indipendence complexes of $G_1$ and $G_2$ are non-isomorphic.   
\end{example}

\noindent
The following is a natural question in relation to \cref{prop:lattices}.

\begin{question}\label{qs:lattices}
  For finite groups $G_1$ and $G_2$, can it happen that $\Sigma(G_1)$ and $\Sigma(G_2)$ are isomorphic even if $G_1$ and $G_2$ do not have isomorphic subgroup lattices? 
\end{question}

\noindent
Since groups with isomorphic independence complexes have the same order, we only consider pairs of groups with this property. Note in fact that there exist finite groups of different orders with isomorphic subgroup lattices, e.g.\ $G_1=C_3\times C_3$ and $G_2=\Sym(3)$.
We conclude the section with a last question that is motivated by \cref{ex:strong} and \cref{prop:lattices}.

\begin{question}\label{qs:strong}
      For finite groups $G_1$ and $G_2$, can it happen that $\Sigma(G_1)$ and $\Sigma(G_2)$ are isomorphic, but $\tilde\Sigma(G_1)$ and $\tilde\Sigma(G_2)$ are not? 
\end{question}

\noindent
We remark that, should the answer to \cref{qs:lattices} be negative in the sense that there even exists an index-preserving lattice isomorphism, then \cref{prop:lattices} would imply a negative answer to \cref{qs:strong}, too. We show that \cref{qs:lattices} and \cref{qs:strong} both have negative answers when $G_1$ is a finite abelian group; cf.\ \cref{thm:ind}. Computational evidence seems to suggest that the answers should be negative in general.

\section{Complexes of abelian groups}\label{sec:ind}

\noindent
In this section, we answer \cref{qs:lattices} in the negative in the case where one of the groups involved is an abelian group. We first settle the case of groups of prime power order, cf.\ \cref{thm:ind-p} and \cref{cor:indp}. We leverage these results to prove \cref{thm:ind} and \cref{cor:ind}. 
We recall that, two finite groups having isomorphic (strong) independence complexes have the same order: in particular, if for some prime number $p$ one is a $p$-group, then so is the other.

\begin{lemma}\label{lem:abel-basic}
Let $p$ be a prime number and let $G_1$ and $G_2$ be finite $p$-groups. 
Assume that $G_1$ is a abelian and that $\Sigma(G_1)\cong\Sigma(G_2)$ or $\tilde \Sigma(G)\cong\tilde\Sigma(G)$.  Then the following hold:
\begin{enumerate}[label=$(\arabic*)$]
    \item\label{it:omega1} If $H\leq G_2$ is such that $\dG(H)=2$, then $|\Omega_1(H)|\leq p^2$;
    \item\label{it:omega1.5} $\Omega_1(G_2)$ is elementary abelian. 
\end{enumerate}
\end{lemma}
\begin{proof}
In this proof, for a subset $X$, we will refer to $X$ as (strongly) independent for arguments that work both for $X$ considered as an independent set or a strongly independent set. Recall that every strongly independent subset is in particular independent and also that in abelian $p$-groups, strong independence and independence coincide; cf.\ \cref{sec:graphs}. 

\noindent 
\ref{it:omega1}
Let $H\leq G_2$ be such that $\dG(H)=2$.
Then $G_2$ is noncyclic and so $G_1$ is also noncyclic.
Since $G_1$ is abelian, if $x,y,z\in G_1$ with $\{x,y\}$ independent and 
$z\in \Omega_1(G_1)$, then either $z\in\gen{x,y}$ or $\{x,y,z\}$ is itself independent. 
It follows that, given a $1$-simplex $\{x,y\}$ in $\Sigma(G_1)=\tilde\Sigma(G_1)$, there are precisely $p^2-1$ vertices $z$ of order $p$ in $\Sigma(G_1)$ such that $\{x,y,z\}$ is not a simplex. As a consequence of \cref{prop:iso-orders}\ref{it:ord3}, the same must be true in $\Sigma(G_2)$ or in $\tilde\Sigma(G_2)$, depending on the simplex that is considered.
A minimal generating set $\{a,b\}$ of $H$ is a (strongly) independent subset of $G_2$ and so there are precisely $p^2-1$ elements $c$ of order $p$ in $G_2$ such that $\{a,b,c\}$ is not (strongly) independent. It follows that $H$ has at most $p^2-1$ elements of order $p$ and so $|\Omega_1(H)|\leq p^2$. 

\ref{it:omega1.5} Let $x$ and $y$ in $G_2$ be of order $p$: we prove that $[x,y]=1$. To this end, define $L=\gen{x,y}$ so that, with $L$ in the role of $H$ in \ref{it:omega1}, the subgroup $\Omega_1(L)$ has order dividing $p^2$. In particular $\Omega_1(L)$ is abelian and, since $x,y\in\Omega_1(L)$, the claim is proven.
\end{proof}

\begin{lemma}\label{lem:comm}
Let $p$ be a prime number and let $G_1$ and $G_2$ be finite $p$-groups. 
Assume that $G_1$ is abelian and that $\Sigma(G_1)\cong\Sigma(G_2)$ or $\tilde\Sigma(G_1)\cong\tilde\Sigma(G_2)$. Let $H\leq G_2$, $y\in G_2$, and $z\in\Omega_1(G_2)$. Then the following hold:
\begin{enumerate}[label=$(\arabic*)$]
    \item\label{it:comm0}  $[y,z]\in\gen{y}\gen{z}$;
    \item\label{it:comm} $\gen{H\cup\{z\}}=H\gen{z}$;
\end{enumerate}
\end{lemma}

\begin{proof}
\ref{it:comm0}
We have $[y,z]=y^{-1}z^{-1}yz\in\Omega_1(\gen{y,z})$ and, if $y=1$, then the claim is clearly true. Assume now that $y\neq 1$ and let $y_1\in\gen{y}$ have order $p$. Then \cref{lem:abel-basic} yields $\gen{z}=\gen{y_1}$ or $\Omega_1(\gen{z,y})=\gen{z}\gen{y_1}$. In both cases, $[y,z]\in\gen{z}\gen{y_1}\subset \gen{z}\gen{y}$. 

\ref{it:comm} The inclusion from right to left is clear so we consider the other one. To this end, let $x\in\gen{H\cup\{z\}}$ and let $x_1,\ldots,x_t\in H\cup\{z\}$ be such that $x=x_1\cdots x_t$. We show by induction on $t$ that there are $h\in H$ and $\zeta\in\gen{z}$ such that $x=h\zeta$. When $t=0$ the claim is clear so we assume $t>0$ and let $h\in H$ and $\zeta\in\gen{z}$ be such that $x_1\cdots x_{t-1}=h\zeta$. It follows that $x=h\zeta x_t$. If $x_t\in\gen{z}$ we are clearly done, so we assume $x_t\in H$. It then follows from \ref{it:comm0} that 
\[
x=h \zeta x_t=h x_t [x_t,\zeta^{-1}]\zeta \in H\gen{x_t}\gen{\zeta^{-1}}\gen{\zeta}\in H\gen{\zeta} \subseteq H\gen{z}. 
\]
This concludes the proof. 
\end{proof}

\begin{lemma}\label{lem:laffey}
Let $p$ be a prime number and let $G_1$ and $G_2$ be finite $p$-groups. 
Assume that $G_1$ is abelian and that $\Sigma(G_1)\cong\Sigma(G_2)$. If $H\leq G_2$, then $\dG(H)\leq \dG(\Omega_1(H))$.
\end{lemma}

\begin{proof}
Let $H$ be a subgroup of $G_2$.
If $G_1$ is cyclic, then the statement is clearly true; we assume therefore that $G_1$ is not cyclic.
Write $\dG(H)=t$ and $H=\gen{y_1,\ldots,y_t}$. Then $\{y_1,\ldots,y_t\}$ is a simplex in $\Sigma(G_2)$ corresponding to a simplex $\{x_1,\ldots,x_t\}$ of $\Sigma(G_1)$. Assume that $z$ is an element of order $p$ of $G_2$ that does not belong to $H$: we claim that $\{y_1,\ldots,y_t,z\}$ is independent. For a contradiction, assume that this is not the case and, since $z$ does not belong to $H$, up to reordering, we have $y_1\in\gen{y_2,\ldots,y_t,z}$. It follows then from \cref{lem:comm}\ref{it:comm} that there are $\tilde{h}\in\tilde{H}=\gen{y_2,\ldots,y_t}$ and $\zeta\in\gen{z}$ such that $y_1=\tilde{h}\zeta$. As a result, $\zeta\in H\cap\gen{z}=\{1\}$  and so $\{y_1,\ldots,y_t\}$ is not independent. Contradiction. 

The last argument, combined with \cref{lem:abel-basic}\ref{it:omega1.5} and \cref{prop:iso-orders}\ref{it:ord1}, shows that, writing 
\[
d=\dG(\Omega_1(G_2)) \ \ \textup{and}\ \ r=\dG(\Omega_1(H))
\]
there are $p^d-p^r$ elements of $G_2$ of order $p$ that are independent from $\{y_1,\ldots,y_t\}$. It follows from \cref{prop:iso-orders}\ref{it:ord3} that there are $p^d-p^r$ elements of order $p$ in $G_1$ that are independent from $\{x_1,\ldots,x_t\}$. The group $G_1$ being abelian, this means that 
\[
p^d-p^r\leq |\Omega_1(G_1)|-|\Omega_1(\gen{x_1,\ldots,x_t})|=p^d- |\Omega_1(\gen{x_1,\ldots,x_t})|
\]
from which we derive that $\dG(H)=t=\dG(\Omega_1(\gen{x_1,\ldots,x_t}))\leq r=\dG(\Omega_1(H))$.
\end{proof}

\begin{remark}\label{rmk:laffey}
    When $p$ is odd and $\Omega_1(H)$ is abelian, \cref{lem:laffey} could be also deduced from  \cite[Cor.~3]{Laf73}. 
\end{remark}

\noindent
The following results are extracted from the classification of minimal metacyclic $p$-groups given in \cite[Thm.~3.2]{Bla61}; see also \cite[Thm.~2]{ZAX06}. 

\begin{lemma}\label{lem:min-nonmeta}
Let $p$ be a prime number and let $K$ be a finite $p$-group with $\dG(K)=2$. Assume that all proper subgroups of $K$ are metacyclic, but $K$ itself is not. Then $p>2$ and one of the following holds:
\begin{enumerate}[label=$(\arabic*)$]
    \item\label{it:min2} $K$ is isomorphic to the Heisenberg group of order $p^3$;
    \item\label{it:min3} $K$ is isomorphic to \textup{\texttt{SmallGroup}(81,10)}.
\end{enumerate}
\end{lemma}

\begin{lemma}\label{lem:min-nonmeta-2}
Let $K$ be a finite $2$-group such that all proper subgroups of $K$ are metacyclic, but $K$ itself is not. Then one of the following holds:
\begin{enumerate}[label=$(\arabic*)$]
    \item\label{it:2min0} $K$ is elementary abelian of order $8$.
    \item\label{it:2min1} $K$ contains a subgroup isomorphic to the quaternion group $Q_8$;
    \item\label{it:2min2} $K$ is isomorphic to \textup{\texttt{SmallGroup}(32,32)}.
\end{enumerate}
\end{lemma}

\begin{proposition}\label{prop:noQ8}
    Let $G_1$ and $G_2$ be finite $2$-groups. 
Assume that $G_1$ is abelian and that $\tilde\Sigma(G_1)\cong\tilde\Sigma(G_2)$. Then $G_2$ does not have subgroups isomorphic to $Q_8$.  
\end{proposition}

\begin{proof}
Assume for a contradiction that $X$ is a subgroup of $G_2$ that is isomorphic to the quaternion group and note that $\Phi(X)=\mho_1(X)=\ZG(X)$. Let now $z$ be an element of order $2$ of $G_2$ that does not belong to $X$ and define $K=\gen{X\cup\{z\}}$. By \cref{lem:comm}\ref{it:comm} we have that $K=X\gen{z}$ and $K$ has order $32$. In particular, $X$ is normal in $K$ and so $z$ maps to an element $\varphi_z$ of $\Aut(X)$ of order dividing $2$. It follows that $z$ centralizes $X$ or $|\varphi_z|=2$.
 Moreover, it follows from \cref{lem:abel-basic}\ref{it:omega1.5} that $\Omega_1(K)$ is abelian. 
Intersecting all conditions, a search in the Small Group Library of GAP shows that $K$ is necessarily isomorphic to the direct product of $X$ and $\gen{z}$, and so $z$ centralizes $X$. 

Again by \cref{lem:abel-basic}\ref{it:omega1.5} we have that $\Omega_1(G_2)$ is elementary abelian and we write $|\Omega_1(G_2)|=2^t$. From the argument above, we derive that $\Omega_1(G_2)$ centralizes $X$ and so $G_2$ contains at least $(8-2)2^{t-1}=3\cdot 2^t$ elements of order $4$ whose square equals the unique element of order $2$ of $X$. Thanks to \cref{prop:iso-orders}\ref{it:ord1}, we have $|\Omega_1(G_1)|=|\Omega_1(G_2)|=2^t$ and, by \cref{prop:iso-orders}\ref{it:ord2}, also that, in $G_1$, there are at least $3\cdot 2^t$ elements whose square is the same element of order $2$. This is a contradiction to the fact that $G_1$ is abelian and so squaring is a homomorphism.  
\end{proof}

\begin{proposition}\label{prop:powerful}
    Let $p>2$ be a prime number and let $G_1$ and $G_2$ be finite $p$-groups. 
Assume that $G_1$ is abelian and that $\Sigma(G_1)\cong\Sigma(G_2)$ or $\tilde\Sigma(G_1)\cong\tilde\Sigma(G_2)$. Then $G_2$ is powerful. 
\end{proposition}

\begin{proof}
Thanks to \cref{prop:iso-orders}\ref{it:ord2}, there is a bijection between the sets 
$$A=\{x^p\mid x\in G_1\} \ \ \textup{ and }\ \ B=\{z^p\mid z\in G_2\}.$$ Write $|G_1|=|G_2|=p^n.$ Since $A=\Phi(G_1),$
we have that $|B|=|A|=p^{n-\dG(G_1)}.$  This yields in particular
$|B|=|\Phi(G_2)|$. Since $B\subseteq \Phi(G_2),$ we conclude that $B=\Phi(G_2)=\mho_1(G_2).$ 
\end{proof}

\begin{lemma}\label{lem:2meta}
    Let $G_1$ and $G_2$ be finite $2$-groups. 
Assume that $G_1$ is abelian and that $\Sigma(G_1)\cong\Sigma(G_2)$ or $\tilde\Sigma(G_1)\cong\tilde\Sigma(G_2)$. If $K$ is a subgroup of $G_2$ with $\dG(K)=2$, then $K$ is metacyclic. 
\end{lemma}

\begin{proof}
Let $K$ be a subgroup of $G_2$ such that $\dG(K)=2$.
For a contradiction, assume that $K$ is not metacyclic and let $J\leq K$ be of minimal order with this property. It follows that $J$ is non-trivial and that every proper subgroup of $J$ is metacyclic. 
Since every subgroup $L$ of $K$ satisfies $\Omega_1(L)\leq \Omega_1(K)$, \cref{lem:abel-basic}\ref{it:omega1} ensures that $|\Omega_1(L)|\leq 4$ and so, as a consequence of \cref{lem:min-nonmeta-2}, the group $J$ is not abelian. 
If $\Sigma(G_1)\cong\Sigma(G_2)$, then \cref{lem:laffey} implies that $\dG(J)\leq \dG(\Omega_1(K))=2$, contradicting \cref{lem:min-nonmeta}.

 We deduce that $\tilde\Sigma(G_1)\cong\tilde\Sigma(G_2)$. 
Then \cref{lem:min-nonmeta-2} combined with \cref{prop:noQ8} yields that $J\cong\texttt{SmallGroup}(32,32)$. We represent $J$ as in \cite[Thm.~2(5)]{ZAX06}:
\[
J=\gen{
a,b,c \mid a^4=b^4=1, \ c^2=a^2b^2, \ [a,b]=b^2, \ [a,c]=a^2, \ [b,c]=1
}.
\]
Then in $J$ the following are satisfied:
    \begin{enumerate}[label=$(\roman*)$]
        \item\label{it:J1} $\Omega_1(J)=\mho_1(J)=\gamma_2(J)=\Phi(J)=\ZG(J)$ and $|\mho_1(J)|=4$;
        \item\label{it:J2} $c^2=a^2b^2=(ab)^2b^2$ and $\Omega_1(J)=\gen{b^2, c^2}$.
    \end{enumerate}
    Let now $z$ be an element of order $2$ of $G_2$ that does not belong to $J$. Define $L=\gen{J\cup\{z\}}$ and so \cref{lem:comm}\ref{it:comm} ensures that $L=J\gen{z}$ and $L$ has order $64$. In particular, $J$ is normal in $L$ and so $z$ maps to an element $\varphi_z$ of $\Aut(J)$ with $\varphi_z^2=1$. We claim that $\varphi_z=1$ and so that $z$ centralizes $J$. To this end, we recall that every $2$-generated subgroup $H$ of $L$ (and thus of $G_2$) satisfies $|\Omega_1(L)|\leq 4$. This holds in particular for every $H=\gen{x,z}$ with $x\in J\setminus\Phi(J)$, and so \ref{it:J1} yields that $\varphi_z(x)\in\gen{x}$. In other words, if $x\in J\setminus\Phi(J)$, then $\varphi_z(x)=x$ or $\varphi_z(x)=x^3$. We start by showing that $\varphi_z(b)=b$. Assume this is not the case and that $\varphi_z(b)=b^3$. Then $\varphi_z(a)=a$ or $\varphi_z(a)=a^3$ and so we have two possibilities: 
    \begin{itemize}
        \item $\varphi_z(ab)=ab^3=(ab)b^2$, or
        \item $\varphi_z(ab)=a^3b^3=(ab)a^2b^2=(ab)c^2.$ 
    \end{itemize}
    In both cases, from \ref{it:J2} we derive that $\varphi_z(ab)\in\gen{ab}$  implies that $\Omega_1(J)=\gen{b^2,c^2}\subseteq \gen{ab}$, which is a contradiction. So we have proven that $\varphi_z(b)=b$. With similar arguments one shows that $\varphi_z(a)=a$ and $\varphi_z(c)=c$. We derive thus that $z$ centralizes $J$. 
    
    Thanks to \cref{lem:abel-basic}\ref{it:omega1.5}, recall now that $\Omega_1(G_2)$ is elementary abelian and write $|\Omega_1(G_2)|=2^t$. By \ref{it:J1}, we have that $t\geq 2$ and it follows from the discussion above that $\Omega_1(G_2)$ centralizes $J$. In particular the number of elements of $G_2$ whose square is in $J\setminus\{1\}$ is lower-bounded by $(32-4)2^{t-2}=7\cdot 2^{t-1}$. It follows from \cref{prop:iso-orders}\ref{it:ord1} that $|\Omega_1(G_1)|=2^t$ and, from \cref{prop:iso-orders}\ref{it:ord2}, that there are, in $G_1$, at least $7\cdot 2^{t-1}$ elements whose square lives in a set of cardinality $3$. This contradicts the fact that squaring is a homomorphism in $G_1$.
\end{proof}

\begin{proposition}\label{prop:2meta}
    Let $p$ be a prime number and let $G_1$ and $G_2$ be finite $p$-groups. 
Assume that $G_1$ is abelian and that $\Sigma(G_1)\cong\Sigma(G_2)$ or $\tilde\Sigma(G_1)\cong\tilde\Sigma(G_2)$. If $K$ is a subgroup of $G_2$ with $\dG(K)=2$, then $K$ is metacyclic.  
\end{proposition}

\begin{proof}
Let $K$ be a subgroup of $G_2$ such that $\dG(K)=2$.
Since $K$ is minimally $2$-generated, $G_2$ is not cyclic and therefore neither is $G_1$.
As a consequence of \cref{lem:2meta}, we assume that $p$ is odd. Since every subgroup $L$ of $K$ satisfies $\Omega_1(L)\leq \Omega_1(K)$, \cref{lem:abel-basic}\ref{it:omega1} together with \cref{rmk:laffey} ensure that $\dG(L)\leq \dG(\Omega_1(K))=2$.
For a contradiction, assume that $K$ is not metacyclic and, without loss of generality, assume also that $K$ is of minimal order with this property. It follows that $K$ is non-trivial and that every proper subgroup of $K$ is metacyclic. Then, \cref{lem:min-nonmeta}\ref{it:min2} combined with \cref{lem:abel-basic}\ref{it:omega1} yields that $p=3$ and $K\cong\texttt{SmallGroup(81,10)}$. 
We represent $K$ as in \cite[Thm.~2(3)]{ZAX06} (where $a$ and $b$ correspond to $s$ and $s_1$ from \cite[Thm.~3.2(iii)]{Bla61}):
\[
K=\gen{a,b,c \mid a^9=b^9=c^3=1, \ a^3=b^{-3}, \ [c,b]=1, \ [b,a]=c, \ [c,a]=a^3}.
\]
Now $G_2$ is powerful, thanks to \cref{prop:powerful}, and therefore
$\mho_2(G_2)=\{x\in G_2 \mid x^9=1\}$.  \cref{prop:iso-orders}\ref{it:ord3} yields then that it is not restrictive to assume that $G_2$ has exponent $9$, so we do. Thanks to $G_2$ being powerful, cubing defines a homomorphism $\rho:G_2\rightarrow\mho_1(G_2)$ and, since $\mho_1(K)=\gen{\rho(a)}$, we have that $\rho^{-1}(\mho_1(K))=\Omega_1(G_2)\gen{a}$. In particular $K$ is contained in $\Omega_1(G_2)\gen{a}$ and Dedekind's Lemma, combined with $G_2$ being powerful, implies 
\[
K=K\cap (\Omega_1(G_2)\gen{a})= \gen{a}(K\cap\Omega_1(G_2))=\gen{a}\Omega_1(K).
\]
This yields a contradiction, because $|\gen{a}\Omega_1(K)|=27\neq 81$.
\end{proof}

\noindent
The following is a summary of the discussion around \cite[Thm.~3.1]{LM87}.

\begin{theorem}\label{thm:modular}
    Let $p$ be a prime number and let $G$ be a finite $p$-group. The following are equivalent:
    \begin{enumerate}[label=$(\arabic*)$]
        \item The group $G$ is modular.
        \item All subgroups of $G$ are powerful. 
        \item All $2$-generated subgroups of $G$ are powerful.
        \item All $2$-generated subgroups of $G$ are metacyclic.
    \end{enumerate}
\end{theorem}

\begin{theorem}\label{thm:ind-p}
Let $G_2$ be a finite $p$-group. The following are equivalent:
\begin{enumerate}[label=$(\arabic*)$]
\item\label{it:abp1} There exists a finite abelian $p$-group $G_1$ such that $\Sigma(G_1)\cong \Sigma(G_2).$
\item\label{it:abp1.5} There exists a finite abelian $p$-group $G_1$ such that $\tilde\Sigma(G_1)\cong\tilde\Sigma(G_2).$
\item\label{it:abp2} $G_2$ is modular and nonhamiltonian.
\item\label{it:abp3} There exists a finite abelian  $p$-group $G_1$ such that $G_1$ and $G_2$ have isomorphic subgroup lattices.
\end{enumerate}
\end{theorem}
\begin{proof}
\noindent \ref{it:abp1}-\ref{it:ab1.5} $\implies$ \ref{it:abp2}. It follows from  \cref{thm:modular} that $G_2$ is modular if and only if all 2-generated subgroups of $G_2$ are metacyclic.  So we apply \cref{prop:2meta} to deduce that $G_2$ is modular.  Moreover, \cref{lem:laffey} and \cref{prop:noQ8} imply that $G_2$ has no subgroup isomorphic to the quaternion group, and therefore $G_2$ is not hamiltonian.


\noindent  \ref{it:abp2} $\implies$ \ref{it:abp3}.
This follows from \cite[Theorem 2.5.10]{schmidt}.

\noindent  \ref{it:abp3} $\implies$ \ref{it:abp1}-\ref{it:abp1.5}.
This follows from \cref{cor:lattices}.
\end{proof}

\noindent
The following is \cite[Thm.~14]{Iwa41}; see also Theorem 2.3.1 in \cite{schmidt}. Combined with this result, \cref{thm:ind-p} yields a full classification of prime power order groups whose independence complex is isomorphic to that of an abelian group. 

\begin{theorem}
    Let $p$ be a prime number and let $G$ be a finite $p$-group. Then the following are equivalent: 
    \begin{enumerate}[label=$(\arabic*)$]
        \item $G$ is modular nonhamiltonian.
        \item There exist a normal abelian subgroup $A$ of $G$, an element $b\in G$ and a positive integer $s$ such that the following hold:
        \begin{itemize}
            \item $G=A\gen{b}$, 
            \item for all $a\in A$, one has $b^{-1}ab=a^{1+p^s}$, and
            \item if $p=2$, then $s\geq 2$.
        \end{itemize}
    \end{enumerate}
\end{theorem}

\begin{corollary}\label{cor:indp}
        Let $p$ be a prime number and let $G_1$ and $G_2$ be finite $p$-groups. 
Assume that $G_1$ is abelian. Then the following are equivalent:
\begin{enumerate}[label=$(\arabic*)$]
    \item\label{it:corp1} $\Sigma(G_1)\cong\Sigma(G_2)$.
    \item\label{it:corp1.5} $\tilde\Sigma(G_1)\cong\tilde\Sigma(G_2)$.
    \item\label{it:corp2} $G_1$ and $G_2$ have isomorphic subgroup lattices.
\end{enumerate}
\end{corollary}

\begin{proof}
    \ref{it:corp1} $\implies$ \ref{it:corp2}. Thanks to \cref{thm:ind-p}, there exists an abelian $p$-group $G_3$ such that $G_2$ and $G_3$ have isomorphic subgroup lattices. It follows from \cref{cor:lattices} that $\Sigma(G_1)\cong\Sigma(G_2)\cong\Sigma(G_3)$ and so 
    that $G_1$ and $G_3$ have isomorphic power graphs. Since power graphs are isomorphism invariants for the family of finite abelian groups, cf.\ \cite[Thm.~1]{cpg1}, we derive that $G_1\cong G_3$ and so $G_1$ and $G_2$ have isomorphic subgroup lattices. The proof of \ref{it:corp1.5} $\implies$ \ref{it:corp2} is analogous.

\noindent
    \ref{it:corp2} $\implies$ \ref{it:corp1}-\ref{it:corp1.5}. This is \cref{cor:lattices}.
\end{proof}



\begin{theorem}\label{thm:ind}
Let $G_2$ be a finite group. The following are equivalent:
\begin{enumerate}[label=$(\arabic*)$]
    \item\label{it:ab1} There  exists a finite abelian group $G_1$ such that $\Sigma(G_1)\cong \Sigma(G_2).$
     \item\label{it:ab1.5} There  exists a finite abelian group $G_1$ such that $\tilde\Sigma(G_1)\cong\tilde\Sigma(G_2).$
    \item\label{it:ab2} $G_2$ is nilpotent
    and its Sylow subgroups are modular and nonhamiltonian.
  \item\label{it:ab3}  There exists a finite abelian group $G_1$ such that $G_1$ and $G_2$ have the same order and isomorphic subgroup lattices.
\end{enumerate}
\end{theorem}
\begin{proof}
\ref{it:ab1} $\implies$ \ref{it:ab2}. 
Suppose $\Sigma(G_2)\cong \Sigma(G_1)$, with $G_1$ abelian.  Then $G_2$ is nilpotent, by \cref{prop:nilpotent}\ref{it:propS1}.  If $G_1$ is cyclic then $G_1\cong G_2.$
Otherwise, by \cref{prop:iso-general}, the isomorphism $\Sigma(G_1)\rightarrow\Sigma(G_2)$ induces, for every prime $p$,
an isomorphism between $\Sigma(P_1)$ and $\Sigma(P_2)$, where 
$P_1$ and 
$P_2$ are, respectively, the Sylow $p$-subgroups of $G_1$ and $G_2.$ Thanks to \cref{thm:ind-p}, the group $P_2$ is modular and nonhamiltonian.

\noindent \ref{it:ab1.5} $\implies$ \ref{it:ab2}. Suppose $\tilde\Sigma(G_2)\cong \tilde\Sigma(G_1)$, with $G_1$ abelian.  Then $G_2$ is nilpotent, by \cref{prop:nilpotent}\ref{it:propS2},
and, by \cref{prop:iso-general-tilde}, the isomorphism $\tilde\Sigma(G_1)\rightarrow\tilde\Sigma(G_2)$ induces, for every prime $p$,
an isomorphism between $\tilde\Sigma(P_1)$ and $\tilde\Sigma(P_2)$, where 
$P_1$ and 
$P_2$ are, respectively, the Sylow $p$-subgroups of $G_1$ and $G_2.$ Thanks to \cref{thm:ind-p}, the group $P_2$ is modular and nonhamiltonian.

\noindent \ref{it:ab2} $\implies$ \ref{it:ab3}. 
There exists a finite abelian group $G_1$ such that the subgroup lattices of $G_1$ and $G_2$ are isomorphic thanks to \cite[Thm.~2.5.10]{schmidt}. The group $G_1$ can be taken so that $|G_1|=|G_2|$; see \cite[Thm.~2.5.9]{schmidt}, on which the proof of \cite[Thm.~2.5.10]{schmidt} is based.

\noindent \ref{it:ab3} $\implies$ \ref{it:ab1}-\ref{it:ab1.5}.
Fix $G_1$ as in \ref{it:ab3}: we show that $\Sigma(G_1)\cong \Sigma(G_2)$. By \cref{prop:nilpotent}\ref{it:propL1}, the group $G_2$ is itself nilpotent and  so \cite[Thm.~4.2.7]{schmidt} ensures that,
if $p$ is a prime number and $P$ and $Q$ are Sylow $p$-subgroups of $G_1$ and $G_2$ respectively, then $P$ and $Q$ have isomorphic subgroup lattices. It follows from \cref{thm:ind-p} that $\Sigma(P)\cong\Sigma(Q)$. 
Calling $P_1,\ldots,P_s$ and $Q_1,\ldots,Q_s$ the Sylow subgroups of $G_1$ and $G_2$ respectively, we have isomorphisms $f_i:\Sigma(P_i)\rightarrow \Sigma(Q_i)$.
Using now the fact that $G_1$ and $G_2$ are nilpotent, 
the bijection 
\[
G=P_1\times\ldots\times P_s\longrightarrow G_2=Q_1\times \ldots \times Q_s, \quad (x_1,\ldots,x_s) \longmapsto (f_1(x_1), \ldots,f_s(x_s))
\]
induces an isomorphism 
$\Sigma(G_1)\rightarrow\Sigma(G_2)$. The proof for strong independence complexes is analogous.
\end{proof}

\noindent
In the language of \cite[Thm.~7]{Su56}, the groups from \cref{thm:ind}\ref{it:ab2} are called quasi-Hamiltonian. 
The following corollaries give negative answers to \cref{qs:lattices} and \cref{qs:strong} for complexes over abelian groups.

\begin{corollary}\label{cor:ind}
        Let $G_1$ and $G_2$ be finite groups. 
Assume that $G_1$ is abelian. Then the following are equivalent:
\begin{enumerate}[label=$(\arabic*)$]
    \item\label{it:cor1} $\Sigma(G_1)\cong\Sigma(G_2)$.
    \item\label{it:cor2} There exists an index-preserving isomorphism $\mathcal{L}(G_1)\rightarrow\mathcal{L}(G_2)$.
\end{enumerate}
\end{corollary}

\begin{proof}
    \ref{it:cor1} $\implies$ \ref{it:cor2}. If $G_1$ is cyclic then $G_1\cong G_2$ and so \ref{it:cor2} is easily seen to hold. Assume now that $G_1$ is not cyclic. By \cref{prop:nilpotent}\ref{it:propS1}, the group $G_2$ is nilpotent and $\Sigma(G_1)\cong\Sigma(G_2)$ ensures, thanks to \cref{prop:iso-general}, for every prime $p$, that $\Sigma(P_1)$ and $\Sigma(P_2)$ are isomorphic, where
$P_1$ and 
$P_2$ are, respectively, the Sylow $p$-subgroups of $G_1$ and $G_2$. Thanks to \cref{cor:indp}, the groups $P_1$ and $P_2$ have isomorphic subgroup lattices. By lifting the isomorphisms corresponding to the different prime numbers, we obtain an index-preserving isomorphism $\mathcal{L}(G_1)\rightarrow\mathcal{L}(G_2)$

\noindent
    \ref{it:cor2} $\implies$ \ref{it:cor1}. The proof of \ref{it:ab3} $\implies$ \ref{it:ab1} in \cref{thm:ind} actually covers this stronger statement.
\end{proof}

\begin{corollary}\label{cor:general-to-star}
             Let $G_1$ and $G_2$ be finite groups. 
Assume that $G_1$ is abelian and that $\Sigma(G_1)\cong\Sigma(G_2)$. Then $\tilde\Sigma(G_1)\cong\tilde\Sigma(G_2)$.
\end{corollary}

\begin{proof}
    Combine \cref{cor:ind} and \cref{prop:lattices}.
\end{proof}

\section{When independent sets are strongly independent}\label{sec:ind-strongind}

\noindent
In this section we give a complete answer to \cref{qs:Q3}, i.e.\ we classify the finite groups $G$ for which $\Sigma(G)=\tilde\Sigma(G)$. Recall that this last condition is equivalent
to \eqref{eq:Q3}: 
\[
 H\leq K\leq G \quad \Longrightarrow \quad \mG(H)\leq \dG(K),
\]
where $\dG(G)$ and $\mG(G)$ denote the minimum and the maximum cardinality of a minimal generating set of $G$, respectively.
We will use this fact without further mention.  

\subsection{Monotone groups and the basis property}

\begin{definition}
Let $G$ be a finite group. Then 
\begin{itemize}
    \item $G$ is a \emph{$\B$-group} if $\dG(G)=\mG(G)$.
    \item $G$ has the \emph{basis property} if every subgroup of $G$ is a $\B$-group.
    \item $G$ is \emph{monotone} if $H\leq K\leq G$ implies that $\dG(H)\leq \dG(K)$.
\end{itemize}
\end{definition}

\noindent
We refer the reader to \cite{AK14} for $\mathcal{B}$-groups, to \cite[App.~A]{AK14}
(see also \cite{McdBQ11}) for groups with the basis property, and to \cite{CM12,Mann05,Mann11} for monotone groups.

\begin{proposition}\label{prop:Q3toMonBasis}
Let $G$ be a finite group. The following are equivalent:
\begin{enumerate}[label=$(\arabic*)$]
    \item $\Sigma(G)=\tilde\Sigma(G)$.
    \item $G$ is a monotone group with the basis property.  
\end{enumerate}
\end{proposition}

\begin{proof}
    Assume first that $G$ satisfies $\Sigma(G)=\tilde\Sigma(G)$. Taking $H=K$ in \eqref{eq:Q3} one immediately derives that $G$ has the basis property. Now for every subgroup $H$ of $G$ one has $\dG(H)=\mG(H)$ and so \eqref{eq:Q3} rewrites as
    \[
    H\leq K\leq G \quad \Longrightarrow \quad \dG(H)\leq \dG(K)
    \]
    so $G$ is monotone. 

    Assume now that $G$ is monotone with the basis property and let $H\leq K\leq G$. Then, $G$ being monotone, we have $\dG(H)\leq \dG(K)$ and, $H$ being a $\B$-group, we deduce $\mG(H)\leq \dG(K)$. The choice of $H$ and $K$ being arbitrary, $G$ satisfies \eqref{eq:Q3}.
\end{proof}

\begin{lemma}\label{lem:Q3quotients}
    Let $G$ be a finite group with $\Sigma(G)=\tilde\Sigma(G)$ and let $H$ and $N$ be subgroups of $G$, with $N$ normal. Then $\Sigma(H)=\tilde\Sigma(H)$ and $\Sigma(G/N)=\tilde\Sigma(G/N)$. 
\end{lemma}

\begin{proof}
By \cref{prop:Q3toMonBasis}, a finite group satisfying \eqref{eq:Q3} is the same as a monotone group with the basis property. The last properties are easily seen to be preserved by subgroups and are preserved by quotients thanks to \cite[Prop.~1.1]{AK14} and \cite[Prop.~1]{Mann11}.
\end{proof}

\noindent
The following is a consequence of \cref{prop:Q3toMonBasis} and \cite[Thm.~1.6]{AK14}.

\begin{proposition}
    Let $G$ be a finite nilpotent group satisfying $\Sigma(G)=\tilde\Sigma(G)$. Then $G$ is a monotone group of prime power order.
\end{proposition}

\noindent
For $p$ an odd prime number, the monotone $p$-groups are classified by Mann; cf.\ \cite{Mann05, Mann11}. The monotone $2$-groups have been classified by Crestani and Menegazzo in \cite{CM12}.

\subsection{The non-nilpotent case}

\noindent
This section is dedicated to the proof of the following theorem: 

\begin{theorem}\label{thm:non-nilp}
 Let $G$ be a finite non-nilpotent group. The following are equivalent: 
 \begin{enumerate}[label=$(\arabic*)$]
     \item\label{it:main1} $\Sigma(G)=\tilde\Sigma(G)$.
     \item\label{it:main2} $G=PQ$ is a Frobenius group such that:
     \begin{enumerate}[label=$(\alph*)$]
        \item $P$ is a normal abelian $p$-subgroup of $G$, where $p$ is a prime;
        \item $Q$ is a cyclic $q$-subgroup of $G$, where $q\neq p$ is a prime;
        \item if $\alpha$ is a generator of $Q$, then exactly one of the following holds:  
       \begin{itemize}
           \item 
           there exists $m\in\Z$ coprime to $p$ such that, for every $x\in P$, one has $x^{\alpha}=x^m$;
           \item\label{it:main2c} $P$ is homocyclic with $\dG(P)=2$ and $|\alpha|$ does not divide $p-1$.
       \end{itemize}
    \end{enumerate}
 \end{enumerate}
\end{theorem}

\begin{remark}\label{rmk:chars}
    Let $G$ be a finite group satisfying $\Sigma(G)=\tilde\Sigma(G)$. Thanks to \cref{thm:non-nilp}, the group $G$ is isomorphic to a semidirect product $P\rtimes_\varphi Q$ where $P$ is a $p$-group, $Q$ is a $q$-group, and the action $\varphi:Q\rightarrow\Aut(P)$ of $Q$ on $P$ is faithful. In particular, $Q$ can be viewed as a subgroup of $\Aut(P)$. Moreover, since the order of $Q$ is coprime to $p$, the canonical map $\Aut(P)\rightarrow\Aut(P/\Phi(P))$ induces an isomorphism of $Q$ with a subgroup $\overline{Q}$ of $\Aut(P/\Phi(P))$. If $\omega:\F_p^\times \rightarrow\Z_p^\times$ denotes the Teichm\"uller character and $P$ is abelian, we write $\chi:\F_p^\times \rightarrow\Aut(P)$ for the composition of $\omega$ with the map giving the $\Z_p$-module structure of $P$. Then condition \ref{it:main2c} can then be interpreted in the following way: 
\begin{itemize}
    \item the order of $Q$ divides $p-1$ and $\varphi(Q)\subseteq \chi(\F_p^\times)$; or  
    \item the order of $Q$ does not divide $p-1$ and there exists a positive integer $n$ such that $P\cong C_{p^n}\times C_{p^n}$.
\end{itemize}
For $P$ abelian, in the first case $\varphi(Q)$ consists of automorphisms of the form $x\mapsto x^m$, where $m$ is an integer coprime to $p$.
In the second case we still have that $|Q|=|\overline{Q}|$ divides $|\GL(2,p)|=p(p-1)^2(p+1)$.
\end{remark}

\noindent
In the following result, which is  \cite[Cor.~A.1]{AK14}, an \emph{$\F_p[Q]$-section}
of a $p$-group $P$, acted upon by another group $Q$, is an $\F_p[Q]$-module $H/N$, where $N\leq H\leq P$ are $Q$-invariant subgroups such that $N$ is normal in $H$ and $N$ contains $\Phi(H)$.

\begin{proposition}\label{prop:AK}
    Let $G$ be a finite non-nilpotent group. The following are equivalent: 
 \begin{enumerate}[label=$(\arabic*)$]
     \item $G$ has the basis property.
     \item $G=PQ$ where:
     \begin{enumerate}[label=$(\alph*)$]
        \item $P$ is a normal $p$-subgroup of $G$, where $p$ is a prime;
        \item $Q$ is a cyclic $q$-subgroup of $G$, where $q\neq p$ is a prime;
        \item\label{it:AK4} if $\Tilde{Q}$ is a non-trivial subgroup of $Q$ and $S$ is an $\F_p[\tilde{Q}]$ section of $P$, then the action of $\tilde{Q}$ on $S$ is faithful and $S$ is isomorphic to a direct sum of isomorphic copies of one simple module.
    \end{enumerate}
 \end{enumerate}
\end{proposition}


\begin{proposition}\label{prop:powers}
    Let $p, q$ be prime numbers. Let $P$ be a finite abelian $p$-group and let $Q=\gen{\alpha}$ be a $q$-subgroup of $\Aut(P)$ such that there exists $m\in\Z$ coprime to $p$ with the property that, for every $x\in P$, one has $x^{\alpha}=x^m$. Then $G=P\rtimes Q$ satisfies $\Sigma(G)=\tilde\Sigma(G)$. 
\end{proposition}

\begin{proof}
Write, for simplicity, $G=PQ$ where the action of $Q$ on $P/\Phi(P)$ is through power automorphisms. Since $q$ divides $p-1$,  the primes $p$ and $q$ are distinct. Relying on \cref{prop:Q3toMonBasis}, we show that $G$ is monotone and has the basis property. The second is guaranteed by \cref{prop:AK}\ref{it:AK4}, so we show monotonicity. 
To this end, note that $\dG(G)=\dG(P)+1$.
Let now $H\leq K$ be subgroups of $G$. Write $H=P_HQ_H$ where $P_H=H\cap P$ and $Q_H$ is a $q$-group; in the same way write $K=P_KQ_K$. Clearly $P_H\leq P_K$ and, $P$ being abelian, $\dG(P_H)\leq \dG(P_K)$. Moreover, if $K$ is contained in $P$, then so is $H$ and so $\dG(H)=\dG(P_H)\leq \dG(P_K)=\dG(K)$. If, on the other hand, $K$ is not contained in $P$, then $\dG(K)=\dG(P_K)+1\geq \dG(P_H)+1\geq \dG(H)$. This proves monotonicity and thus the proof is complete. 
\end{proof}

\noindent
The following result is a consequence of the Cayley-Hamilton Theorem.

\begin{lemma}\label{lem:CH}
    Let $p$ be a prime number and let $A$ be an element of $\GL(2,p)$. Then the following hold:
    \begin{enumerate}[label=$(\arabic*)$]
        \item\label{it:CH1} if $|A|$ divides $p-1$, then $A$ has at least one eigenvalue in $\F_p$.
        \item\label{it:CH2} if $|A|$ does not divide $p(p-1)$, then the characteristic polynomial of $A$ is irreducible.  
    \end{enumerate}
\end{lemma}

\begin{proposition}\label{prop:coprime}
   Let $p,q$ be prime numbers. Let $P$ be a finite abelian $p$-group and let $Q$ be a cyclic subgroup of $\Aut(P)$ of order not dividing $p(p-1)$.  Assume that $P$ is homocyclic and that $\dG(P)=2$. 
     Then $G=P\rtimes Q$ satisfies $\Sigma(G)=\tilde\Sigma(G)$. 
\end{proposition}

\begin{proof}
Write, for simplicity, $G=PQ$ and note that $q$ is odd and different from $p$. Moreover, since $P$ is abelian, we have $\Phi(P)=\mho_1(P)$ and, $P$ being homocyclic with $\dG(P)=2$, for $n$ a non-negative integer we have
\begin{equation}\label{eq:mho}
|\mho_n(P):\mho_{n+1}(P)|=\begin{cases}
p^2 & \textup{if } p^{n+1}\leq \exp(P), \\
1 & \textup{otherwise}.
\end{cases}
\end{equation}
Write $Q=\gen{\alpha}$ and $\overline{Q}=\gen{\overline{\alpha}}$ for the image of $Q$ under the canonical map $\Aut(P)\rightarrow\Aut(P/\Phi(P))\cong\GL(2,p)$. Since $|Q|$ does not divide $p(p-1)$, by \cref{lem:CH}\ref{it:CH2}, the $\F_p[Q]$-module $P/\Phi(P)$ is simple. Moreover, since $Q$ acts  on $P$ by automorphisms and $p$-th powering induces a surjective homomorphism $\mho_n(P)/\mho_{n+1}\rightarrow\mho_{n+1}(P)/\mho_{n+2}$, the equality in \eqref{eq:mho} ensures, for every nonnegative integer $n$ with $p^n\leq \exp(P)$, that  $\mho_n(P)/\mho_{n+1}(P)$ is a simple $\F_p[Q]$-module. 
The same applies to every nontrivial subgroup $\tilde{Q}$ of $Q$. 

Let now $\tilde{Q}\leq Q$ be non-trivial and let $H$ be a subgroup of $P$ that is $\tilde{Q}$-stable. Assume that $H\neq\{1\}$ and let $n$ be the smallest non-negative integer such that $H$ is contained in $\mho_n(P)$ but is not contained in $\mho_{n+1}(P)$. The induced action of $\tilde{Q}$ on $\mho_n(P)/\mho_{n+1}(P)$ being irreducible, we deduce that $\mho_n(P)=H\mho_{n+1}(P)$ and so \eqref{eq:mho} yields that $H=\mho_n(P)$. 

This shows in particular that the $\F_p[\tilde{Q}]$-sections of $P$ are precisely the quotients of the form $\mho_n(P)/\mho_{n+1}(P)$ and so \cref{prop:AK} yields that $G$ has the basis property. 

In view of \cref{prop:Q3toMonBasis}, we are only left with showing that $G$ is monotone. For this, 
let $H\leq K$ be subgroups of $G$. Write $H=P_HQ_H$ where $P_H=H\cap P$ and $Q_H$ is a $q$-group; in the same way write $K=P_KQ_K$. Then $P_H\leq P_K$ and, $P$ being abelian, $\dG(P_H)\leq \dG(P_K)$. Moreover, if $K$ is contained in $P$, then so is $H$ and so $\dG(H)=\dG(P_H)\leq \dG(P_K)=\dG(K)$. If, on the other hand, $K$ is not contained in $P$, then $\dG(K)=\dG(P_K)+1\geq \dG(P_H)+1\geq \dG(H)$. This proves monotonicity and thus the proof is complete.
\end{proof}

\begin{proposition}\label{prop:Q3->d2} 
    Let $G$ be a finite non-nilpotent group such that $\Sigma(G)=\tilde\Sigma(G)$. Write $G=PQ$ as in \cref{prop:AK} and $Q=\gen{\alpha}$. Then one of the following holds: 
    \begin{enumerate}[label=$(\arabic*)$]
        \item\label{it:d2.1} $\dG(P)=\dG(G)=2$ and $P/\Phi(P)$ is a simple $\F_p[Q]$-module.
        \item\label{it:d2.2}  $P$ is abelian and there exists $m\in\Z$ coprime to $p$ such that, for every $x\in P$, one has $x^{\alpha}=x^m$. 
    \end{enumerate}  
\end{proposition}

\begin{proof}
We prove \ref{it:d2.1} and \ref{it:d2.2} together. Let $V$ be a simple $\F_p[Q]$-module and $t$ a positive integer such that, as $\F_p[Q]$-modules, $P/\Phi(P)$ and $V^t$ are isomorphic; the existence of $V$ and $t$ is guaranteed by \cref{prop:AK}\ref{it:AK4}.  \cref{prop:Q3toMonBasis} yields that $G$ is monotone, so 
    \[
    t\dG(V)=\dG(P)\leq \dG(G)\leq t+1. 
    \]
    Only two cases can therefore occur:
    \begin{enumerate}[label=$(\alph*)$]
        \item $\dG(V)=1$ or
        \item $\dG(V)=2$ and $t=1$. 
    \end{enumerate}
It is clear that, in the second case, $\dG(P)=\dG(P/\Phi(P))=\dG(V)=2$, so we  now consider the first. Assume therefore that $\dG(V)=1$ and write $\overline{\alpha}$ for the non-trivial automorphism of $\overline{P}=P/\Phi(P)$ that is induced by $\alpha$. Let $m\in\Z$ be such that, for every $\overline{x}\in\overline{P}$, one has $\overline{\alpha}(\overline{x})=\overline{x}^m$. Then $m$ does not belong to $p\Z \cup (1+p\Z)$. 

 For a contradiction, we assume now, without loss of generality, that $P$ has class $2$ and that $\gamma_2(P)$ has exponent $p$ (in other words, with the aid of \cref{lem:Q3quotients}, we are replacing $P$ with $P/\gamma_3(P)\mho_1(\gamma_2(P)$)). Then $\Phi(P)$ is central and so the commutator map induces a bilinear map $P/\Phi(P)\times P/\Phi(P)\rightarrow\gamma_2(P)$. We deduce that, for every $x\in\gamma_2(P)$, one has $\alpha(x)=x^{m^2}$.
 Let now $g\in P\setminus \Phi(P)$ and note that $H=\gen{g}\gamma_2(P)$ is abelian and $Q$-stable. Since $g\in P\setminus \Phi(P)$ and 
 $m\not\equiv m^2\bmod p$
 , we derive a contradiction from 
  \cref{prop:AK}\ref{it:AK4}. 
 We have thus proven that $P$ is abelian. 
\end{proof}

\begin{example}\label{ex:Q8}
    Let $P=Q_8$ and recall that $\Aut(P)\cong \Sym(4)$. Let $\alpha\in\Aut(P)$ correspond to $(1\ 2\ 3)$ via a chosen isomorphism and let $\overline{\alpha}$ be the image of $\alpha$ in $\Aut(P/\Phi(P))$. Define $Q=\gen{\alpha}$ and $G=P\rtimes Q$. The characteristic polynomial of $\overline{\alpha}$ is $x^2+x+1\in\F_2[x]$, which is irreducible. In particular, the only $Q$-stable subgroups of $P$ are ${1}$, $P$, and $\Phi(P)$. However $\alpha$ restricted to $\Phi(P)$ is the identity map and so $G$ does \underline{not} satisfy \eqref{eq:Q3}. 
\end{example}

\noindent
The following is \cite[Thm.~2]{Luc13}. Recall that a chief factor $X/Y$ of a finite group $G$ is \emph{non-Frattini} if $X/Y$ is not a subgroup of $\Phi(G/Y)$. 

\begin{proposition}\label{prop:chief}
    Let $G$ be a finite solvable group. Then $\mG(G)$ coincides with the number of non-Frattini factors in a chief series of $G$.
\end{proposition}

\begin{lemma}\label{lem:p-odd}
     Let $G$ be a finite non-nilpotent group such that $\Sigma(G)=\tilde\Sigma(G)$ and write $G=PQ$ as in \cref{prop:AK}. Assume that $P$ has class $2$ and that $\ZG(P)$ has exponent $p$. Then $|P|=p^3$ and $p$ is odd. 
\end{lemma}

\begin{proof}
Since $P$ is not abelian, \cref{prop:Q3->d2}\ref{it:d2.1} yields that $\dG(P)=\dG(G)=2$. Moreover, the exponent of $\gamma_2(P)$ being $p$, the subgroup $\Phi(P)$ is central and so $\gamma_2(P)$ has order $p$. Since both $\ZG(P)$ and $\gamma_2(P)$ are $\F_p[Q]$-modules, there exists a submodule $M$ of $\ZG(P)$ such that $\ZG(P)=\gamma_2(G)\oplus M$. 
We claim that $M=\{1\}$. For a contradiction, assume that $M$ is not trivial. Then \cref{prop:Q3toMonBasis}, together with
\cref{prop:chief}, yields that $2=\dG(G)\geq \dG(\ZG(P)Q)=\mG(\ZG(P)Q)\geq 3$; a contradiction. 
We have proven that $M=\{1\}$ and so $\ZG(P)=\gamma_2(P)$. Since $\Phi(P)$ is central and contains $\gamma_2(P)$, we get that $\gamma_2(P)=\Phi(P)=\ZG(P)$ and so $|P|=p^3$. 

In conclusion, there are only two non-abelian groups of order $8$, up to isomorphism, being $D_8$ and $Q_8$. The group $D_8$ cannot figure as a normal $2$-Sylow of  a group satisfying \eqref{eq:Q3} because $\Aut(D_8)\cong D_8$, while $Q_8$ is excluded by \cref{ex:Q8}.
\end{proof}

\noindent
The following is a consequence of \cite[Theorem 4]{Mann05}.

\begin{lemma}\label{lem:mann}
   Let $p$ be a prime number and let $P$ be a finite monotone $p$-group such that $\dG(P)\leq 2$. Then $P$ is metacyclic or $\ZG(P)$
has exponent $p$.
\end{lemma}



\begin{proposition}\label{prop:metacyclic}
 Let $G$ be a finite non-nilpotent group such that $\Sigma(G)=\tilde\Sigma(G)$ and write $G=PQ$ as in \cref{prop:AK}. Assume that $P$ has class $2$. Then $P$ is metacyclic.
\end{proposition}

\begin{proof}
Since $P$ is not abelian, Proposition \ref{prop:Q3->d2}\ref{it:d2.1} ensures $\dG(P)=2$ and that $P/\Phi(P)$ is a simple $\F_p[Q]$-module. 
Moreover, $P$ is monotone, so \cref{lem:mann} yields that $P$ is metacyclic or $\ZG(G)$
has exponent $p$. Assume, for a contradiction, that $P$ is not metacyclic. Then, by \cref{lem:p-odd}, the prime $p$ is odd and $P$ has order $p^3$ and so, $P$ not being metacyclic, the exponent of $P$ is $p$. 
Thanks to \cref{prop:AK}\ref{it:AK4}, the subgroup $\gamma_2(P)$ is not fixed by $Q$ and $|Q|$ divides $p-1=|\Aut(\gamma_2(P))|$.
Now \cref{lem:CH}\ref{it:CH1} ensures that the elements of $Q$ fix a $1$-dimensional subspace of  $P/\Phi(P)$ and so $P/\Phi(P)$ is not simple. Contradiction.
\end{proof}

\noindent
The following result is a consequence of \cite[Sec.~1]{Men93}.

\begin{proposition}\label{prop:aut-meta}
    Let $p$ be an odd prime number and let $P$ be a finite metacyclic $p$-group. 
    Let $\alpha\in\Aut(P)$ have order coprime to $p$. Then one of the following holds:
    \begin{itemize}
        \item $P$ is homocyclic with $\dG(P)=2$; or
        \item the order of $\alpha$ divides $p-1$.
    \end{itemize}
\end{proposition}

\noindent
We conclude the section by finally giving the proof of \cref{thm:non-nilp}.

\begin{proof}[Proof of \cref{thm:non-nilp}]
We prove the equivalence of the two statements.

\ref{it:main2} $\implies$ \ref{it:main1}. This is the combination of \cref{prop:powers} and \cref{prop:coprime}.

\ref{it:main1} $\implies$ \ref{it:main2}.
 We borrow the notation from \cref{prop:AK}, which also ensures that the group $G$ is Frobenius and the action of $Q$ on $P/\Phi(P)$ is faithful. 
 
We now show that $P$ is abelian. Assume for a contradiction that this is not the case. Then, thanks to \cref{lem:Q3quotients}, we assume without loss of generality that $P$ has class $2$. \cref{prop:metacyclic} yields that $P$ is metacyclic.  

Assume first that $p=2$ and let  $K$ be the unique subgroup of order $2$ of $\gamma_2(P)$. Then $K$ is characteristic in $P$ and, in particular, it is $Q$-stable. The order of $K$ being $2$, it follows that $Q$ acts trivially on $K$. Contradiction to \cref{prop:AK}\ref{it:AK4}. 

    Assume now that $p$ is odd. Since $P$ is not abelian, 
 \cref{prop:aut-meta}
yields that $|Q|$ divides $p-1$ and \cref{prop:Q3->d2}\ref{it:d2.1} that $P/\Phi(P)$ is a simple $\F_p[Q]$-module. Contradiction to \cref{lem:CH}\ref{it:CH1}.

    We have thus proven that $P$ is abelian and proceed with showing that one of the conditions in \cref{thm:non-nilp}\ref{it:main2c} holds. To this end, let $\alpha$ be a generator of $Q$ and assume that $\alpha$ is not a power automorphism on $P/\Phi(P)$. It follows from \cref{prop:Q3->d2} that $\dG(P)=2$ and that $P/\Phi(P)$ is a simple $\F_p[Q]$-module. In particular, \cref{lem:CH}\ref{it:CH1} implies that the order of $\alpha$ does not divide $p-1$. 
    Moreover, \cref{prop:AK}\ref{it:AK4} ensures that for every non-negative integer $n$, the quotient $\mho_n(P)/\mho_{n+1}(P)$ is trivial or a simple $\F_p[Q]$-module, where the action of $Q$ is faithful. Assuming that $\mho_n(P)\neq\mho_{n+1}(P)$, from the fact that $|Q|$ does not divide $p-1$ we derive that $|\mho_n(P):\mho_{n+1}(P)|=p^2$. We conclude that $P$ is homocyclic.  
\end{proof}

\end{document}